\documentclass[11pt]{aims}
\usepackage{amsmath, amssymb, mathrsfs}
  \usepackage{paralist}
  \usepackage{graphics} 
  \usepackage{epsfig} 
\usepackage{graphicx} 
 \usepackage{epstopdf}
 \usepackage[colorlinks=true]{hyperref}
 \hypersetup{urlcolor=blue, citecolor=blue}
 \usepackage[toc,page]{appendix}
\usepackage{chngcntr}
\usepackage{hyperref}

\usepackage{titlesec}
\titleformat{\section}{\Large\bfseries}{\thesection.}{4pt}{}
\titleformat{\subsection}{\large\bfseries}{\thesection.\arabic{subsection}.}{4pt}{}
\titleformat{\subsubsection}{\bfseries}{\thesection.\arabic{subsection}.\arabic{subsubsection}.}{4pt}{}
\titleformat*{\paragraph}{\bfseries}
\titleformat*{\subparagraph}{\bfseries}
\usepackage[margin=1in]{geometry}

  \def\currentyear{}
   \def\currentmonth{}
    \def\ppages{}
     \def\DOI{}

\newtheorem{theorem}{Theorem}[section]

\newtheorem{lemma}[theorem]{Lemma}
\newtheorem{proposition}[theorem]{Proposition}
\theoremstyle{definition}
\newtheorem{definition}[theorem]{Definition}
\newtheorem{remark}[theorem]{Remark}

\newcommand{\R}{\mathbb{R}}
\newcommand{\N}{\mathbb{N}}

\numberwithin{equation}{section}

\title[Blowup solutions for the shadow Gierer-Meinhardt system]{
Blowup solutions for the shadow limit model of a singular Gierer-Meinhardt  system with critical parameters}

\author[G. K. Duong,    T. E. Ghoul, N. I. Kavallaris and H. Zaag ]{}

\thanks{\today}
\begin{document}
\maketitle

\centerline{G.  K. Duong$^{*}$,  T. E. Ghoul$^{*}$,  N. I. Kavallaris$^{**}$   and H. Zaag$^{***}$} 
\medskip
{\footnotesize
  \centerline{ $^{*}$ New York University in Abu Dhabi, Department of Mathematics }
   \centerline{$^{**}$  University of Chester, Department of Mathematical and Physical Sciences,  Chester, UK }
   \centerline{$^{***}$ Universit\'e Sorbonne Paris Nord,  LAGA CNRS (UMR 7539), F-93430, Villetaneuse,
    France}
}

\bigskip
\begin{center}\thanks{\today}\end{center}

\begin{abstract} 
 We consider a nonlocal parabolic PDE, which may be regarded as the standard semilinear heat equation with power nonlinearity, where the nonlinear term is divided by some Sobolev norm of the solution. Unlike the  earlier work in  \cite{DNZarxiv2020}  where we consider a subcritical regime of parameters, we focus here on the critical regime, which is much more complicated. Our main result concerns the construction of a blow-up solution with the description of its asymptotic behavior. Our method relies on a formal approach, where we find an approximate solution. Then, adopting a rigorous approach, we linearize the equation around that approximate solution, and reduce the question to a finite dimensional problem. Using an argument based on index theory, we solve that finite-dimensional problem, and derive an exact solution to the full problem. 
We would like to point out that our constructed solution has a \textit{new} blowup speed with a $\log$ correction term, which makes it different from the speed in the subcritical range of parameters and the standard heat equation.


\end{abstract}

\section{Introduction}
In the current work,  we consider the following nonlocal parabolic equation 
\begin{equation}\label{equa-u-non-local}
\left\{\begin{array}{rcl}
   \partial_t u    &=&  \Delta u  -  u + {\displaystyle\frac{u^p}{ \left(\mathop{\,\rlap{-}\!\!\int}\nolimits_\Omega    u^r dr  \right)^\gamma }}\quad\text{in}\quad   \Omega \times  (0,T), \\
\frac{ \partial u}{ \partial \nu}    &=&  0   \text{ on  }      \partial \Omega \times  (0,T)\\
u(0) & = & u_0 \geq 0 \quad \text {in } \quad \Omega,
\end{array} \right.
\end{equation}
where $\Omega$ is a bounded domain in $\R^N$ with  smooth boundary, and the parameters satisfy $ p>1, \gamma >0$ and $ r >0$. The nonlocal problem \eqref{equa-u-non-local} arises  from the shadow system of a singular Gierer-Meinhardt system, as a limiting case when inhibitor's diffusion dominates activator's one, cf. \cite{KSN17, NKMI2018}. Its dynamical behavior was first considered in \cite{KSN17}, where in particular global-in-time existence as well as blowup results were derived according to the range of the involved parameters $p,r$ and $\gamma.$ The case of an isotropically evolving domain was also considered in \cite{KBM20}, where an analytical and numerical study was delivered. In the aforementioned works, it was pointed out that for the limiting problem  \eqref{equa-u-non-local}, {\it diffusion-driven (Turing) instability} occurs, an intriguing phenomenon which was introduced in the seminal paper \cite{TPTRS1952}, for a specific range of the parameters. In particular,  the observed {\it diffusion-driven (Turing) instability} for \eqref{equa-u-non-local} occurs in the form of {\it diffusion-driven blowup} under the {\it Turing condition}
\begin{equation}\label{tcon}
p-r\gamma<1,
\end{equation}
cf. \cite{KBM20, KSN17, NKMI2018}. However, in those works, only a rough form of the {\it Turing instability (blowup) patterns} is presented and only in the case of a sphere; it is based on known results, cf. \cite{MZgfa98, HVcpde92, HVaihn93}  for the standard heat equation
\begin{equation}\label{classical-heat-equation}
\partial_t u=\Delta u+u^p,
\end{equation}
which coincides with \eqref{equa-u-non-local} when $\gamma=0$. 

\medskip

Let us mention here some important results  involving to  finite time blowup solutions to \eqref{classical-heat-equation}. First, we mention the papers \cite{BKnon94, MZdm97, MZnon97}, where the authors constructed  blowup solutions obeying the following profile: 
\begin{equation}\label{blowup-profile-classical-heat-equation}
    \left\|(T-t)^{\frac{1}{p-1}} u(x,t) - \left(p-1 +\frac{(p-1)^2}{4p} \frac{|x|^2}{(T-t)|\ln(T-t)|} \right)^{-\frac{1}{p-1}}  \right\|_{L^\infty} \le \frac{C}{1+\sqrt{|\ln(T-t)|}}.
\end{equation}
We also observe that 
$$ u(0,t) \sim \kappa (T-t)^{-\frac{1}{p-1}} :=\psi(t),$$
with  $\kappa   = (p-1)^{-\frac{1}{p-1}}$ and   $\psi(t)$  exactly solves 
$$   \partial_t \psi(t) = \psi^p(t),  $$
which is the ODE associated to \eqref{classical-heat-equation}. 
In the literature, if a solution of \eqref{classical-heat-equation} blows up at time $T>0$ and satisfies 
\[
\forall t\in [0,T), \| u(.,t) \|_{L^\infty} \le C \psi(t),
\]
then, it is called a ``Type I" blowup solution. If not, then, it is called a ``Type II". Clearly, the solution shown in \eqref{blowup-profile-classical-heat-equation} is of Type I.
%
%
This classification can naturally be extended to the general equation 
$$ \partial_t u =\Delta u + F(u(t)). $$ 

Regarding the construction of Type  I blowup solutions to \eqref{classical-heat-equation}, the authors in \cite{BKnon94} (also in \cite{MZdm97})   used      an important method  consisting in  two main steps: first, \textit{a reduction to a finite-dimensional problem}, then  \textit{a topological argument based on index theory} to solve the finite-dimensional problem.    We also  mention that the   method  has been proven robust in a lot of situations  such as  \cite{MZnon97} for the  reconnection of vortex;  in \cite{TZpre15, DNZtunisian-2017, NZsns16} for perturbated nonlinear source terms;  in \cite{DNZMAMS20, DNZIHPNA21, MZjfa08, NZ2017}  for blowup solutions to complex Ginzburg-Landau equations; in  \cite{NZcpde15, DJFA2019, DJDE2019}  for   complex-valued heat equations which has no  variational structure; in parabolic systems as in \cite{GNZsysparabolic2016}; and also in \cite{DZM3AS19}   for MEMS models.
 
 \medskip
 In addition to that,  a huge literature has been devoted in the last 20 years to the construction of solutions of PDEs with prescribed behavior, beyond the case of parabolic equations such as: Type I anisotropic blowup for the heat equation by Merle et al \cite{MRSIMRNI20}; Type II blowup for the heat equation by del Pino et al \cite{PMWAMS19,PMWAPDE20,PMWQYARXIV20,PMWZDCDS20}, Schweyer \cite{SJFA12}, Collot \cite{CAPDE17}, Merle et al \cite{CMRJAMS20}, Harada \cite{HAIHPANL20,HAPDE20}, Seki \cite{SJDE20}; blowup for nonlinear Schr\"odinger equation  by Merle \cite{MCMP90}, Martel and Merle \cite{MRASENS18}, Merle et al \cite{MRRCJM15,MRSDMJ14,MRRSARXIV20}, Rapha\"el and Szeftel \cite{RSCMP09};  blowup for wave equations by C\^ote and Zaag \cite{CZCPAM13}, Ming et al \cite{MRTSJMA15}, Collot \cite{CPMJEMS20}, Hillairet and Rapha\"el \cite{HRAPDE12}, Krieger et al \cite{KSTDM09,KSTIM08},  Ghoul et al \cite{GINJDE18}, Rapha\"el and Rodnianski \cite{RRPMIHES12}, Donninger and Sch\"orkhuber \cite{DSCMP16}; blowup for KdV and gKdV by Martel \cite{MAJM05} and C\^ote \cite{CJFA06,CDM07}; Schr\"odinger maps by Merle et al \cite{MRRim13}; heat flow map  by Ghoul et al \cite{GINAPDE19},  Rapha\"el and  Schweyer \cite{RSAPDE14}, D\'avila et al \cite{DPMIM20}; the Keller-Segel system by Ghoul et al \cite{CGMNARXIV20-a,CGMNARXIV20-b}, Schweyer and Rapha\"el \cite{RSma14}, Prandtl's system by Collot et al \cite{CGIMARXIV18}; Stefan problem by Hadzic and Rapha\"el \cite{HRJEMS19}; three-dimensional fluids by Merle et al \cite{MRRSARXIV20}.
 
 \bigskip
Now, we come back to the (nonlocal) equation \eqref{equa-u-non-local} when $\gamma \ne 0$. Recently, inspired by the Type I construction mentioned above for equation \eqref{classical-heat-equation}, the authors of \cite{DNZarxiv2020} constructed a blowup solution to equation \eqref{equa-u-non-local}  via a rigorous analysis,  giving the exact form of the {\it blowup profile}
in  some subcritical regime of the parameters, namely when
\begin{equation}\label{sub-critical}
 \frac{r}{p-1} < \frac{N}{2} \text{ and }   \gamma r \ne p-1,
\end{equation}
being in agreement with {\it Turing condition} \eqref{tcon}. Specifically, under \eqref{sub-critical}, the authors constructed   a blowup solution to \eqref{equa-u-non-local}  with the blowup profile (pattern) as follows
\begin{equation}\label{result-DNZ-21}
    u(x,t) \sim  (\theta^*)^{-\frac{1}{p-1}} (T-t)^{-\frac{1}{p-1} }  \left(   p-1 + \frac{(p-1)^2}{4p} \frac{|x|^2}{(T-t) |\ln(T-t)|}   \right)^{-\frac{1}{p-1}} \text{ as } t \to T,
\end{equation}
 where 
 $$ \theta^*:= \lim_{t \to T}  \left( \mathop{\,\rlap{-}\!\!\int}\nolimits_\Omega    u^r dr \right)^{-\gamma}.$$
Note that if the constant $\theta^*$ is ignored, then the solution has the same structure of   blowup solutions, as the solution constructed  by  \cite{BKnon94} and \cite{MZdm97} for the standard (local) equation \eqref{classical-heat-equation}. 
In fact, the approach in that work imposed a special analysis to control the non-local term and make it converge to a nonzero constant. This way, we reasonably see that in this regime, equation \eqref{equa-u-non-local} will behave like the standard equation \eqref{classical-heat-equation}. In particular, up to some natural scaling, both equations show the same profile 
(see \eqref{blowup-profile-classical-heat-equation} and \eqref{result-DNZ-21}).  Inspired by the analysis in \cite{DNZarxiv2020}, it naturally arises the need for the consideration of another regime where the non-local term has a different limit, in particular it converges towards ``zero", leading hopefully to a different blow-up speed.
The following is our main result:     

 \begin{theorem}\label{theorem-existence}    Let  $\Omega$ be a smooth and bounded domain in $ \mathbb{R}^N$ containing the origin and consider equation \eqref{equa-u-non-local} in the following \textit{critical} regime:
\begin{equation}\label{critical-condition}
    \frac{r}{p-1}  =  \frac{N}{2}\text{ with  } p \ge 3.
\end{equation}
 Then,  there exists $\gamma_0$ small such that for all $ \gamma \in (0,\gamma_0)$, we can construct    initial data  $u_0   \geq  0    $  such that the solution of   \eqref{equa-u-non-local}    blows up in finite time $T(u_0)$,  only  at  the origin. Moreover, we have the following     blowup asymptotics:
 \begin{itemize}
 \item[$(i)$] \textbf{Behavior of $\theta(t)= \left(\mathop{\,\rlap{-}\!\!\int}\nolimits_\Omega    u^r dr  \right)^{-\gamma}$}. It holds that 
\begin{equation}\label{asymptotic-theta-mai-novelty}
\theta(t) = \theta_\infty|\ln(T-t)|^{-\beta} \left( 1+ O\left(\frac{1}{\sqrt{|\ln(T-t)|}} \right) \right), \text{ as } t \to T,  
\end{equation}
where 
$$ \theta_\infty = \left( \frac{|\Omega| \left(1 +\frac{N}{2} \right)}{1- \frac{\gamma N}{2}} 2b^\frac{N}{2}\right)^\frac{\gamma}{1- \frac{\gamma N}{2}},$$
\[
\beta = \left(\frac{N}{2} +1 \right) \frac{\gamma  }{1  -  \frac{\gamma N}{2}},\;\;
 b = \frac{(p-1)^2}{4p} (1  + \beta).
 \]
 \item[$(ii)$] \textbf{The   intermediate blowup profile}.   For  all $ t \in (0,T)$, we have
 \begin{eqnarray}\label{estima-T-t-u-theorem-varphi-0}
 & & \left\|    (T-t)^{\frac{1}{p-1}}  |\ln(T-t)|^{-\nu} u(\cdot  , t  )  -   ( \theta_\infty)^{-\frac{1}{p-1}}    \varphi_0 \left(   \frac{|.|}{ \sqrt{(T-t) |\ln(T-t)|}}  \right) \right\|_{L^\infty(\Omega)} \label{intermediate-profile}\\
  &&   \hspace{7cm}  \leq  \frac{C}{1 + \sqrt{|\ln(T-t)|}} \nonumber.
 \end{eqnarray}
 where
 \[
 \varphi_0 \left(  z \right) =   \left( p-1 + b |z|^2  \right)^{-\frac{1}{p-1}} \text{ and } \nu = \frac{\beta}{p-1}. 
 \]
\item[$(iii)$] \textbf{The final blowup profile}. It holds that  
$ u(x,t)   \to    u^* (x) \in C^2( \Omega  \setminus  \{ 0\})$  as $t\to T $,  uniformly     on  compact sets   of  $\Omega  \setminus  \{ 0\}$.  In particular, we have  
\begin{equation}\label{final-profile}
u^*(x)      \sim   (\theta_\infty)^{ -\frac{1}{p-1}} \left[ \frac{b}{2} \frac{|x|^2}{|\ln|x |}\right]^{ -\frac{1}{p-1}} \left[2|\ln|x|| \right]^\nu \text{ as }  x \to 0.  
\end{equation}
 \end{itemize}
 \end{theorem}
 \begin{remark}[Stability]
 Following the interpretation of $N+1$ parameters for the blowup time and the blowup point, originally done in \cite{MZdm97}, and then applied  in \cite{DNZarxiv2020}, we can prove that behaviors \eqref{intermediate-profile} and \eqref{final-profile} are stable under perturbation of initial data, the readers can find more details in Remark 2.3 of \cite{DNZarxiv2020}.
 \end{remark}
 \begin{remark}[On the smallness  of $\gamma$]  We mention that in the proof of the Theorem, we need      $\gamma$ to be very small. This is due to the limitation of our method, which is inspired by the case of the standard equation \eqref{classical-heat-equation} treated in 
 \cite{BKnon94} and \cite{MZdm97}. However, regardless of the method, we suspect that the  described behavior
  will not occur when $\gamma$ large, since $\theta$ may act against blowup in equation \eqref{equa-u-non-local} in that case.
  
  Furthermore, we would like to point out that the condition on the smallness of parameter $\gamma$ is consistent with analogous conditions on the exponents of non-local terms that guarantee finite-time  blowup for parabolic and hyperbolic non-local problems, for more details see \cite{BL97, NK04, KT04, KN07, KS07, KT02, KT02B, KT06} and references therein. 
\end{remark} 
\begin{remark}[Novelty] The constructed solution in Theorem \ref{theorem-existence} has the following blowup speed
\[
\|u(t)\|_{L^\infty} \sim \kappa(\theta_\infty)^{-\frac 1{p-1}}(T-1)^{-\frac 1{p-1}}|\ln(T-t)|^\nu, \kappa = (p-1)^{-\frac{1}{p-1}},
\]
which is different from the subcritical regime treated in \cite{DNZarxiv2020}, and also from the case of the standard heat equation \eqref{classical-heat-equation}, where no $|\ln(T -t)|$ correction appears. 

 In fact,
our problem is non-local,  as it involves the following term 
$$ \theta (t)   =    \frac{1}{   \left(\mathop{\,\rlap{-}\!\!\int}\nolimits_\Omega    u^r(t) dx \right)^\gamma }$$
and thus we first observe that once $u$ blows up, the non-local integral will  affect the solution's asymptotics. Let us now mention the pioneering papers \cite{DNZarxiv2020} and  \cite{DZM3AS19}, where the authors construct blowing up
and quenching solutions to equations involving non-local terms, and describe their blowing up and quenching profile respectively.
However, in the those works, 
the authors only handled the regime where the non-local term stays away from $0$ and infinity, in the sense that
$$ \theta(t) \to \theta^* >0, \text{ as } t \to T.$$
Then, thanks to a natural scaling, the non-local term has  no big impact  on the solution's blowing up or queching behavior. 
In the current paper, we unveil a 
\textit{new phenomenon}, 
where
$$ \theta(t) \to 0$$
which clearly affects the nonlinear term, in the sense that $$ \theta(t) u^p \ll u^p.$$ 
Naturally, the solution's behavior will be more affected here than in the subcritical case treated in \cite{DNZarxiv2020}. 
In the present paper, we mainly rely on the construction method of \cite{MZnon97} (also \cite{BKnon94} and \cite{MZdm97}), however, we need new ideas to carefully control the behavior of $\theta(t)$, so that it fits the description given in statement $(i)$ of Theorem \ref{theorem-existence}.
\end{remark}

\begin{remark}
Remarkably we should have that
\begin{equation*}
1 -\frac{\gamma N}{2} > 0\implies \gamma<\frac{2}{N}=\frac{p-1}{r}\implies p-r\gamma>1,
\end{equation*} 
and thus we sit in the regime where {\it Turing condition} \eqref{tcon} is not satisfied. Therefore our main result given by Theorem \ref{theorem-existence} 
 describes the occurrence of a {\it reaction driven} blowup.
\end{remark}

\medskip
\noindent
\textbf{The Organization:}  This paper is organized as follows:

- In  Section \ref{section-formal-approach},
we give   a formal approach to  derive  the behavior of $\theta(t)$, together with the blowup profile.

- In Section \ref{section-formulation},  we rigorously formulate the problem.

- In Section \ref{estistence-without-technical},
we give the proof of Theorem \ref{theorem-existence}, assuming some technical results.

- In Sections \ref{section-reduction-finite} and   \ref{proof-initial-data-propo},  we give the proofs of the technical results used in Section \ref{estistence-without-technical}. 

- In Sections \ref{appendix-A}, \ref{appendix-B} and \ref{estimat-pro(x)}, we give some purely technical 
computations, which are useful for the proof.

\section{Formal approach}\label{section-formal-approach}
In this  section, we aim at giving a formal approach which explains how  the blowup  profile  in Theorem \ref{theorem-existence}  is derived.  Firstly, let us denote    
\begin{equation}\label{defi-theta-}
\theta (t)   =    \frac{1}{   \left(\mathop{\,\rlap{-}\!\!\int}\nolimits_\Omega    u^r(x,t) dx \right)^\gamma }. 
\end{equation} 
Henceforth,  we  rewrite   equation  \eqref{equa-u-non-local}  by the following
\begin{equation}\label{equa-main-short}
\partial_t u  = \Delta u - u     +  \theta(t) u^p. 
\end{equation}
Through a formal observation, we focus on    the following three     interesting situations:
  \begin{eqnarray}
  \theta(t)   & \to &  0     \text{ as }   t \to T   \label{case-theta-to-0},\\
  \theta(t)   & \to  &  \theta^* >0   \text{ as }  t \to T  \label{case-theta-to-thata-*},\\
  \theta  (t)  & \to &  +\infty   \text{ as }   t \to T  \label{case-theta-to-infty}.
  \end{eqnarray}
We point out  that \eqref{case-theta-to-infty}  is excluded   by   Theorem 3.1 and Remark 3.2  given in  \cite{KSN17}. Recently, \eqref{case-theta-to-thata-*} was handled  in \cite{DNZarxiv2020} under some sub-critical regimes (see more in \eqref{sub-critical}).  Thus in the current work we provide  analytically a construction of blowup solution,    satisfying \eqref{case-theta-to-0}.   In the following, we aim to deliver a formal study deriving  the proper blowing behavior.
Here we outline our formal approach into two steps:

- Step 1: A prescribed asymptotic of the solution on $\Omega$. More precisely, we   are inspired from   \cite{DNZarxiv2020}, and \cite{DZM3AS19} (see also in \cite{MZnon97}), to    control the solution on  the  three domains under dynamical hypotheses on the non-local term   $\theta(t)$, given in \eqref{assymp-formal-theta}:
\begin{enumerate}
\item[$\bullet$] Blowup region $P_1$ defined by
\begin{equation}\label{P-1}
P_1(t) = \left\{   |x| \le K_0 \sqrt{(T-t)|\ln(T-t)|}   \right\}.
\end{equation} 
\item[$\bullet$]  Intermediate region $P_2(t)$ defined by 
\begin{equation}\label{defi-P-2}
P_2(t) = \left\{  \frac{K_0}{4} \sqrt{(T-t)|\ln(T-t)|}   \le |x| \le \epsilon_0   \right\},
\end{equation}
for some $ \epsilon_0$ small enough.
\item[$\bullet$] Regular domain $P_3$ defined by 
\begin{equation}\label{defi-P-3}
P_3  = \{  x \in \Omega \text{ such that } |x| \ge \epsilon_0   \}.
\end{equation}
\end{enumerate}  
- Step 2:  We give a formal justification to show derived behaviors of  $u$ on  each $P_j,j=1,2,3$ which is adapted to        \eqref{case-theta-to-0}.  
 \subsection{Control of the solution's    asymptotics on $\Omega$ }
 As we mentioned above,  in this paragraph we will provide explicit behavior of $u$ on regions $P_1, P_2$ and $ P_3$.   Now let us assume that $u$ is a blowing up solution  in finite time $T$ at  the origin $  0 \in \Omega$,  and \eqref{case-theta-to-0} is satisfied.
 
 \begin{center}
    \textbf{ Asymptotic of the intermediate profile in region $P_1$  }
\end{center}

\medskip
 \noindent 
 Let us introduce the following similarity variable:
\begin{equation}\label{similarity-variable}
y  = \frac{x}{\sqrt{T-t}},     \quad  s  = -\ln(T-t)  \text{ and  } W (y,s)  = (T-t)^{\frac{1}{p-1}} U(x,t) ,
\end{equation}
with
\begin{equation}\label{defi-U-by-u}
     U(x,t) := \theta(t)^{\frac{1}{p-1}} u(x,t),
 \end{equation}
 which by virtue of \eqref{defi-theta-} entails
 \begin{eqnarray}\label{defi-theta-by-U}
\theta (t)  =     \left(  \def\avint{\mathop{\,\rlap{--}\!\!\int}\nolimits} \avint_\Omega
U^r\,dx  \right)^{ -\frac{\gamma}{1- \frac{r \gamma}{p-1}}}.
 \end{eqnarray}
Next using equation \eqref{equa-u-non-local},  $U$ reads 
\begin{equation}\label{equa-U-theta-'-theta}
\partial_t  U     =  \Delta U          + U^p    +\left( \frac{1}{p-1} \frac{\theta'(t)}{\theta(t)}-1 \right) U,
\end{equation}
where $\theta(t)$ is defined as in \eqref{defi-theta-by-U}.
Using   \eqref{similarity-variable} and  \eqref{equa-U-theta-'-theta}, $W$ solves
\begin{eqnarray}
\partial_s W =  \Delta W  - \frac{1}{2} y \cdot \nabla W  - \frac{W}{p-1}     + W^p + \left( \frac{1}{p-1}\frac{\bar \theta'(s)}{\bar \theta(s)}  - e^{-s} \right)  W ,\label{equa-W-formal-approach}
\end{eqnarray}
where 
\begin{equation}\label{defi-bar-theta-s}
    \bar \theta(s) = \theta(t(s)),\quad s = -\ln (T-t).
\end{equation}
Regarding  $\theta$'s  evolution, we   make a hypothesis      as follows     near the blowup point 
\begin{equation}\label{hypothesis-theta's}
  \frac{\theta'(t)}{\theta(t)}  \ll  U^{p-1},
\end{equation}
then \eqref{equa-U-theta-'-theta}  is considered as a small perturbation to the following
\begin{equation}\label{classcial-heat-equation}
\partial_t U = \Delta U + U^p,
\end{equation}
  which   has been   studied  thoroughly in \cite{BKnon94,  HVasps92, MZdm97}, and the references therein.  The authors in those works constructed a blowup solution to  \eqref{classcial-heat-equation}, satisfying 
$$ U \sim \kappa(T-t)^{-\frac{1}{p-1}} \text{ near the blowup region}\;P_1,$$ 
from which in conjunction with \eqref{hypothesis-theta's} we derive
\begin{equation}\label{rough-assumption-theta-t}
\frac{\theta'(t)}{\theta(t)}  \ll (T-t)^{-1},
\end{equation}
and using the fact that  $ \bar \theta(s) = \theta(t)  $ then \eqref{rough-assumption-theta-t} yields  
\begin{eqnarray*}
 \frac{\bar \theta'(s)}{\bar \theta(s)}  \to 0, \text{ as }   s \to +\infty.
\end{eqnarray*}
We remark that the following situation is impossible:
$$  \left| \frac{\bar \theta'(s)}{ \bar \theta(s)} \right| \le \frac{C}{s^{1 +\delta}}, \text{ for some }  \delta > 0, $$
due to the fact that $\bar \theta \to 0 $ as $ s \to +\infty.$    However, it is hard to classify the asymptotic behavior of this quantity and so we only provide a special case so that   \eqref{case-theta-to-0} is valid. For example, we consider the following situation
\begin{equation}\label{assume-bar-theta-'-theta}
\frac{\bar \theta' (s)}{ \bar \theta(s)} =  -\beta s^{-1} + O(s^{-1-\delta}), 
\end{equation}
for some $\beta, \delta >0$. In particular, \eqref{assume-bar-theta-'-theta}  has a special solution that
\begin{eqnarray*}
\bar \theta(s)  = \theta_\infty s^{-\beta} \left( 1  + O\left(\frac{1}{s^\delta}\right)\right), \text{ as } s \to +\infty,
\end{eqnarray*}
This also  implies 

\begin{equation}\label{assymp-formal-theta}
\theta(t) = \theta_\infty |\ln (T-t)|^{-\beta
}\left(1 + O\left( \frac{1}{|\ln (T-t)|^{\delta}} \right)  \right), \text{ as } t \to T.
\end{equation}

We remark that \eqref{assymp-formal-theta} will rigidly  be justified  in  subsection \ref{subsection-justification-theta}.
Plugging  \eqref{assume-bar-theta-'-theta} into  \eqref{equa-W-formal-approach}, we observe that 
\begin{eqnarray*}
\left( \frac{\bar \theta'}{\bar \theta}  - e^{-s} \right) W = -\frac{\beta}{p-1} \frac{W}{s} + \text{ ``lower order'' }.
\end{eqnarray*}
Hence,  we interested in considering the following proxy equation
\begin{equation}\label{equa-W}
\partial_s W =  \Delta W - \frac{1}{2} y \cdot \nabla W - \frac{W}{p-1}     + W^p -\frac{\beta}{ s(p-1)}W,   
\end{equation}
for all $  (y,s)  \in \Omega_s \times  [-\ln(T-t), +\infty)$ and   $\Omega_s   =  e^{\frac{s}{2}} \Omega $.    

Note that there exists a   space-independent    solution of \eqref{equa-W}  in the following form 
$$ W =    \kappa + O\left(\frac{1}{s} \right),$$
recalling that $\kappa   = (p-1)^{-\frac{1}{p-1}}.$

Now, we consider the linearization around the dominated part $\kappa$ by
$$ \bar W = W -\kappa,$$
  which yields       
\begin{equation}\label{equa-bar-W}
\partial_s   \bar W  =  \mathcal{L} \bar W   + \bar B(\bar W) - \frac{\beta(\kappa +\bar W)}{(p-1)s} ,
\end{equation}
where 
\begin{eqnarray*}
\mathcal{L} &=& \Delta - \frac{1}{2} y \cdot \nabla + Id,\\
\bar B(\bar W) & =&  (\bar W + \kappa)^p - \kappa^p - p \kappa^{p-1} \bar W.
\end{eqnarray*}
In addition to that,   we have   the fact that  for all $|\bar W |  \leq  1$
$$   \left| \bar B(\bar W)  - \frac{p}{ 2 \kappa}  \bar W^2   \right| \leq  C|\bar W|^3.$$
We now assume that  $\bar W $  is radial and   expressed  by  
$$  \bar W(y,s)  = \bar   W_0(s)    + \bar W_2 (s) (|y|^2   - 2 N)$$
(the  idea behind  this argument, i.e. that the solution approaches a radial profile, and $W_2 (|y|^2 -2N)$ is dominated part, can be found in  \cite{TZpre15}).
We  expect  to  have  $\bar W_0, \bar W_2    \to 0, \text{ as } s \to +\infty.$
 Projecting   the above expansion on \eqref{equa-bar-W}, we get  
\begin{eqnarray*}
\bar W_0'(s) & =& \bar  W_0     + \frac{p}{2 \kappa}  (\bar W_0^2  + 8 N\bar W_2^2) -\frac{\beta}{(p-1) s}  (\kappa  +  \bar W_0 ) + O(|\bar W_0|^3 + |\bar W_2|^3) , \\
\bar W_2'(s) & =&    \frac{4 p}{\kappa}  \bar W_2^2   + \frac{p}{\kappa} \bar W_0 \bar W_2   - \frac{\beta}{(p-1) s}   \bar W_2  +  O(|\bar W_0|^3 + |\bar W_2|^3) .
\end{eqnarray*}
Solving     the above   ODE system, we   obtain
\begin{eqnarray*}
\bar W_0 (s)  & =& \frac{\beta \kappa}{(p-1) s}   + o\left(\frac{1}{s}\right), \\
\bar W_2 (s) & =&     - \frac{\kappa}{4 ps } \left( 1  +\beta\right)   + o\left(\frac{1}{s} \right),
\end{eqnarray*}
and thus we establish the following inner expansion
 \begin{eqnarray}
 W(y,s) & =& \kappa -   \frac{k}{4p s} \left(1 +\beta \right)   (|y|^2 -2N)  +    \frac{\beta \kappa}{(p-1) s} +o \left(\frac{1}{s} \right)\nonumber\\
& = & \kappa -   \frac{k}{4p s} \left(1 +\beta \right)   |y|^2  +    \frac{N\kappa}{2ps}\left( 1 + \beta  \right) + \frac{\beta \kappa}{(p-1) s} +o \left(\frac{1}{s} \right).\label{expansion-bar-W}
 \end{eqnarray}
We now study the outer expansion. By the form of the blowup variable 
$$  z =  \frac{y}{\sqrt{s}}, $$
we are motivated to seek for a blowup profile 
 in $z$ as follows  
$$  W(y,s)   =   \varphi_0( z  )    + \text{ lower perturbation}, \text{ as } s \to +\infty. $$ 
 Plugging into \eqref{equa-W}, we derive
\begin{equation}\label{defi-varphi-0}
 \varphi_0 \left(  z \right) =   \left( p-1 + b |z|^2  \right)^{-\frac{1}{p-1}}, 
\end{equation} 
 where    it naturally  requires  $b >0$, since $\varphi_0$ needs  to be global.  In  particular, by matching to inner expansion   \eqref{expansion-bar-W}, we obtain 
 \begin{equation}\label{defi-b}
  b = \frac{(p-1)^2}{4p} \left( 1 +\beta \right). 
 \end{equation}
 Thus,  we deduce  the blowup profile
 \begin{equation}\label{defied-varphi}
 \varphi(y,s) =  \left( p-1 + b \frac{|y|^2}{s} \right)^{-\frac{1}{p-1}}  +  \frac{N\kappa}{2ps}\left( 1 +\beta \right) + \frac{\beta \kappa}{(p-1) s},
 \end{equation}
 which  is   close to  the solution
 $$ W(y,s)  \sim    \varphi(y,s), \text{ as  } s \to +\infty.  $$
 The formal result     is  adapted from  \cite{DZM3AS19, MZdm97, MZnon97}. In particular, in \cite{TZpre15}, the authors  obtained the $W^{1,\infty}$ estimate
 \begin{equation}\label{estimate-W-1-infty}
  \|  W - \varphi\|_{W^{1,\infty}}    \le \frac{C}{1+\sqrt{|\ln (T-t)|}}.
 \end{equation}
 Let us assume that \eqref{estimate-W-1-infty}  hold. Thus, we  derive 
\begin{equation}\label{estima-formal-P-1}
\left\{     \begin{array}{rcl}
& & \left|   (T-t)^{\frac{1}{p-1}} \theta^{\frac{1}{p-1}}(t)    u(x,t) - \left(p-1 + b \frac{|x|^2}{(T-t)|\ln(T-t)|} \right)^{-\frac{1}{p-1}} \right| \le \frac{C}{1+\sqrt{|\ln(T-t)|}},\\[0.2cm]
 & & |\nabla  u(x,t)|  \le    \frac{C (T-t)^{-\frac{1}{p-1} -\frac{1}{2}}  \theta^{-\frac{1}{p-1}}(t)   }{1 + \sqrt{|\ln(T-t)|}},
\end{array}    \right.
\end{equation}
for all $ x \in P_1(t)$.
\medskip
\begin{center}
    \textbf{ Asymptotic of the intermediate profile in region $P_2$:  }\label{formal-rescaled-U-P-2}
\end{center}
In region $P_2$,   we try to control  a  rescaled  version   of $u$. 
Firstly, we define   for all $|x| \le  \epsilon_0 $ with $\epsilon_0$ small  enough, $t(x)$ as the unique solution  of the following equation
\begin{eqnarray}
|x| &=& \frac{K_0}{4} \sqrt{ (T-t(x) ) | \ln(T-t(x))|} \text{ with  } t(x) < T.\label{c4defini-t(x)-}
\end{eqnarray}
Note that,  $t(x)$  is well  defined as long as  $\epsilon_0 $ is  small enough and  we have the following  asymptotic behavior
$$  t(x)  \to  T, \text{ as }  x \to 0.$$
 For convenience, we introduce 
\begin{equation}\label{c4defini-theta}
\varrho (x) = T - t(x),
\end{equation}
so, it follows
$$\varrho(x) \to 0  \text{ as } x  \to 0.$$
Next,  we assume that $u$  well define on $[0,t_1]$   and we  introduce then  the re-scaled function
\begin{equation}\label{c4rescaled-function-U}
\mathcal{U} (x, \xi, \tau) =  \left(T-t (x) \right)^{\frac{1}{p-1}}  \theta^{\frac{1}{p-1}}(t(x))  u \left(X,t \right),
\end{equation}
where
\[
X=  x + \xi \sqrt{ T- t(x)}\mbox{ and }
t=  \varrho(x) \tau  +  t(x).
\]
Note that $\tau$ is considered  to belong in    $  \left[ -\frac{t(x)}{T-t(x)}, \frac{t_1-t(x)}{T-t(x)} \right]$, since  
$t\in [0,t_1]$,
then the problem is well defined.  By \eqref{equa-main-short},   $\mathcal{U}$ satisfies
 \begin{equation*}
    \partial_t  \mathcal{U} =  \Delta_\xi \mathcal{U}   + \tilde \theta(\tau) (\theta (t(x)))^{-1}  \mathcal{U}^p       - \rho (x)\mathcal{U},
\end{equation*}
where  
\begin{equation*}
    \tilde{\theta}(\tau) = \theta(\tau \varrho(x) +t(x)), \text{ and } \rho(x) \text{ defined as in } \eqref{c4defini-theta}.
\end{equation*}
The readers should understand that $\theta(t') =\theta(0)$ if $t'\le 0$.
 \medskip
We now refer to \cite{MZnon97}, (see also \cite{DZM3AS19}, and \cite{DNZarxiv2020})   in which the authors     studied  $\mathcal{U}$'s dynamic  on a small region of the local space $(\xi,\tau)$ defined by
 $$  |\xi| \leq  \alpha_0  \sqrt{|\ln( \varrho(x))|}  \text{ and }   \tau \in \left[ -\frac{  t(x)}{ \varrho(x)}, 1 \right).  $$
 
In particular,   the  key idea is to show that this flatness is preserved for all $\tau \in [0, 1)$ (that is for all $t \in [t(x),T)$), in the sense that the solution does not depend substantially on space.  Using that argument, we  derive that  $\mathcal{U}$ is regarded   as a perturbation  to  $\hat{ \mathcal{U}}(\tau),$ where  $\hat{ \mathcal{U}}(\tau)$  solves the problem
 \begin{equation*}\label{c4equa-hat-mathcal-U}
 \left\{  \begin{array}{rcl}
 \partial_\tau  \hat{\mathcal{U}} (x,\tau)   & = &  \tilde \theta(\tau) \theta^{-1}(t(x))  \hat{ \mathcal{U}}^p (x,\tau),\\
 \hat{\mathcal{U}} (0)  & = &  \left(  p-1  + \displaystyle \frac{(p-1)^2}{4 p} \displaystyle \frac{K_0^2}{16} \right)^{-\frac{1}{p-1}},
\end{array} \right.
 \end{equation*}
and  is explicitly given by 
 \begin{equation}\label{c4defin-hat-mathcal-U-tau}
 \hat{\mathcal{U}} (x,\tau) =  \left(   (p-1) \left(1 -   \int_0^\tau \tilde \theta(\tau') \theta^{-1}(t(x)) d\tau'  \right)  +  \displaystyle \frac{(p-1)^2}{4 p} \frac{K_0^2}{16}  \right)^{-\frac{1}{p-1}}.
 \end{equation}
From   $t(x)$'s  monotonicity,  we have that
 $$ t(x)  \le t\quad\mbox{and so}\quad T-t(x) \ge T-t.  $$
 and   the asymptotic behavior   \eqref{assymp-formal-theta},    which implies       $\theta$ is decreasing. Henceforth, we obtain
 $$ \tilde \theta(\tau') \theta^{-1}(t(x))  \le 1,  \forall  \tau'\in [0,\tau] .$$
 Then, we derive
 $$    |\hat{ \mathcal{U}} (\tau) | \le C ,   $$
 which  implies
 $$ | \mathcal{U}(x,0,\tau)|  \le C.$$
 In addition to that,  and relying on the flatness  of $\hat{\mathcal{U}}$,   it is reasonable to assume the following  estimate on  the gradient 
 $$  \left|  \nabla_\xi \mathcal{U} (x,0,\xi)    \right|  \le  \frac{C}{\sqrt{|\ln \rho(x)|}}. $$
 Finally, we use   Lemma \ref{lemma-t(x)} and  $\theta$'s asymptotic assumed  at  \eqref{assymp-formal-theta} to derive the following $u$'s asymptotic on $P_2$: 
 \begin{equation}\label{estima-P-2-formal-approach}
 \left\{ 
 \begin{array}{rcl}
 \left|  u(x,t) \right| \le   C  \left[ |x|^2 \right]^{-\frac{1}{p-1}} |\ln|x||^{\frac{1}{p-1}}  \theta^{-\frac{1}{p-1}}(t)   ,\\[0.3cm]
 \left| \nabla u (x,t)   \right|  \le  C (|x|^2)^{-\frac{1}{p-1} -\frac{1}{2}}  |\ln|x||^{ \frac{1}{p-1}-\frac{1}{2}} \theta^{-\frac{1}{p-1}}(t).
 \end{array}
 \right.
 \end{equation}
\begin{center}

    \textbf{ Asymptotic profile in the regular region $P_3$:  }
\end{center}
 Using  the well-posedness of  the Cauchy problem  for equation  \eqref{equa-main-short}, we derive the asymptotic profile of  the solution $u$ within that region as a  perturbation of initial data $u(0)$. In particular,  we  obtain some  estimates as follows
 \begin{equation}\label{estima-formal-aproach-P-3}
 \begin{array}{rcl}
  \left|  u (x,t) \right|  + | \nabla u(x,t)|\le C,
 \end{array}
 \end{equation}
for all $ x \in P_3$.

  \subsection{  On dynamical hypothesis of $\theta$ }\label{subsection-justification-theta}
 In the current subsection, we aim to give a  justification of 
ansatz \eqref{assume-bar-theta-'-theta}  by using the formal behavior of the solution in regions $P_j,j=1,2 $ and $3$, sketched in  \eqref{estima-formal-P-1}, \eqref{estima-P-2-formal-approach} and \eqref{estima-formal-aproach-P-3}. To this end we are inspired by  \cite{DNZarxiv2020}, in which  the authors   have proved an explicit  asymptotic   to  $L^k(\Omega)$, and which is here applied to  the critical case $\frac{r}{p-1}  =\frac{N}{2}$,   see more details in Corollary  2.2. The later result   also  motivates us to derive the new blowup speed in the current paper. However, the approach developed in  \cite{DNZarxiv2020}  needs to be  developed further  so we can derive a more  precise behavior on  $\theta $ defined by \eqref{defi-theta-}; that is actually the main strategy applied in the current work.   

We realize that once    \eqref{case-theta-to-0}  occurs, it immediately follows that
$$  \left(\mathop{\,\rlap{-}\!\!\int}\nolimits_\Omega    u^r(x,t) dx \right)^\gamma  \to +\infty, \text{ as } t \to T,$$
provided that $\gamma >0$.  
Multiplying    \eqref{equa-main-short} by $r {u}^{r-1}$  and integrating over $\Omega$ we derive
\begin{eqnarray}
\partial_t \|  u\|^r_{L^r(\Omega)} & = & r \int_{\Omega } \Delta u u^{r-1} dx + r \theta(t) \int_{\Omega}  u^{p-1+r} dx - r \int_\Omega u^r dx  \label{derivation-tilde-u},\\
& = &  (1-r)r \int_\Omega |\nabla u|^2  u^{r-2} dx + r \theta(t) \int_\Omega u^{p-1+r} dx  - r \int_\Omega u^r dx. \nonumber
 \end{eqnarray}
Following a similar tedious calculation as in Section 6 in \cite{DNZarxiv2020} for the case $ \frac{r}{p-1} <\frac{N}{2}$, we use 
  \eqref{estima-formal-P-1}, \eqref{estima-P-2-formal-approach} and  \eqref{estima-formal-aproach-P-3} to obtain the following estimates
\begin{eqnarray}
\left|  \int_\Omega  | \nabla u|^2   u^{r-2} dx \right|     & \le & C (T-t)^{-1} |\ln(T-t)|^{ \frac{N}{2} -1}  \theta^{-\frac{N}{2}}(t)  ,\\
\left|  \int_\Omega u^r dx \right|  & \le &  C |\ln(T-t)|^{\frac{N}{2} +1 } \theta^{-\frac{N}{2}}(t).
\end{eqnarray} 
 It then remains   to   estimate 
\begin{eqnarray*}
 I(t)  &=& r \theta(t) \int_{\Omega }  u^{p-1+r} dx =  \frac{r}{  \left(  \frac{1}{|\Omega|} \int_\Omega  u^r dx \right)^\gamma  }  \int_\Omega  u^{p-1+r} dx \\
 & = & r |\Omega |^\gamma (\| u\|^r_{L^r})^{-\gamma} \int_\Omega u^{p-1+r} dx= r \theta(t) \int_\Omega u^{p-1 +r} dx.
\end{eqnarray*}
Let us  note that    \eqref{estima-formal-P-1} still holds within a larger domain, i.e. 
$$  | x| \le  K_0 \sqrt{(T-t) |\ln(T-t)|}  |\ln(T-t)|^\frac{1}{2},$$
provided that \eqref{estimate-W-1-infty} is valid. 

\medskip
Next, we focus on estimating the following  integral
$$  \int_{\Omega} u^{p-1 +r} dx.        $$
Indeed we decompose it as
\begin{eqnarray*}
\int_{\Omega} u^{p-1 +r} dx & =& \int_{|x| \le K_0 \sqrt{(T-t)|\ln(T-t)|} |\ln(T-t)|^\frac{1}{2}} u^{p-1+r} dx \\
&+&   \int_{  K_0 \sqrt{(T-t)|\ln(T-t)|} |\ln(T-t)|^\frac{1}{2}  \le |x| \le \epsilon_0 } u^{p-1+r} dx \\
& + &  \int_{  |x| \ge \epsilon_0,  x \in \Omega } u^{p-1+r}dx.
\end{eqnarray*}
From \eqref{estima-formal-aproach-P-3}, we derive
$$  \left|   \int_{  |x| \ge \epsilon_0,  x \in \Omega } u^{p-1+r}    \right| \le C.  $$
Besides that,  we  use  \eqref{estima-P-2-formal-approach} to deduce
\begin{eqnarray*}
& &     \int_{  K_0 \sqrt{(T-t)|\ln(T-t)|} |\ln(T-t)|^\frac{1}{2}  \le |x| \le \epsilon_0 } u^{p-1+r} \\
& \le &  C \theta^{ -(1+ \frac{N}{2})  }(t) \int_{  K_0 \sqrt{(T-t)|\ln(T-t)|} |\ln(T-t)|^\frac{1}{2}  \le |x| \le \epsilon_0 }   \left(  \frac{|x^2|}{|\ln |x||}   \right)^{- 1  - \frac{N}{2}} dx\\
  & \le &     C \theta^{ -(1+ \frac{N}{2})  }(t)  |\ln(T-t)|^{\frac{N}{2} -1}.
\end{eqnarray*}
In addition by virtue of   \eqref{estima-P-2-formal-approach}, \eqref{estima-formal-aproach-P-3}  and  the fact that $ \frac{r}{p-1}  =\frac{N}{2}$ and $p\ge 3$, we derive
\begin{eqnarray*}
\int_\Omega u^{p-1 +r}   &=& \theta^{ - 1 - \frac{N}{2}}(t)  \left( \int_{0}^{ K_0 |\ln(T-t)|^\frac{1}{2}}  \varphi_0^{p-1+r} (\xi) \xi^{N-1} d\xi  \right) (T-t)^{-1} |\ln(T-t)|^{ \frac{N}{2}  }\\
& + & O( \theta^{-1 -\frac{N}{2}}(t) (T-t)^{-1} |\ln(T-t) |^{ \frac{N}{2} -\frac{1}{2}}) .
\end{eqnarray*}
On the other hand via the asymptotic behavior  \eqref{assymp-formal-theta},    we obtain  the following
\begin{eqnarray}
\partial_t \| u\|^r_{L^r} & = &  B ( \|u\|^r_{L^r})^{ -\gamma}  (T-t)^{-1} |\ln(T-t)|^{\frac{N}{2} + \beta (\frac{p-1 +r}{p-1})  } \nonumber\\
 & + & O( (T-t)^{-1} |\ln(T-t)|^{\frac{N}{2} + \beta (\frac{p-1 +r}{p-1})   -\frac{1}{2}}, \label{behavior-norn-u-r-derivative}
\end{eqnarray}
where 
\begin{equation}\label{defi-B-constant}
B = (\theta_\infty)^{- (\frac{N}{2} +1)}  r |\Omega|^\gamma \int_0^\infty  \varphi_0^{p-1+r} (\xi) \xi^{N-1} d\xi .   
\end{equation}
This yields
\begin{eqnarray}
 \|u\|^r_{L^r} &=& \left( \frac{(1+\gamma)  B}{1+ \frac{N}{2} +\beta ( \frac{p-1+r}{p-1})} \right)^{\frac{1}{\gamma+1}} | \ln(T-t)|^{\left(1+ \frac{N}{2} +\beta ( \frac{p-1+r}{p-1}) \right) \frac{1}{1+\gamma} } \nonumber\\
 & + & O( |\ln(T-t)|^{ \left(1+ \frac{N}{2} +\beta ( \frac{p-1+r}{p-1}) \right) \frac{1}{1+\gamma} -\frac{1}{2}} ).\label{behavior-norn-u-r}
\end{eqnarray}
Next $ \theta$'s definition, via \eqref{defi-theta-}, yields 
\begin{eqnarray*}
\theta(t)  &=& |\Omega|^\gamma (\|u\|^r_{L^r})^{-\gamma},\\
& =& |\Omega|^\gamma  \left( \frac{(1+\gamma)  B}{1+ \frac{N}{2} +\beta ( \frac{p-1+r}{p-1})} \right)^{\frac{-\gamma}{\gamma+1}} |\ln(T-t)|^{\left(1+ \frac{N}{2} +\beta ( \frac{p-1+r}{p-1}) \right) \frac{-\gamma}{1+\gamma} }\\
& +& O \left( |\ln(T-t)|^{\left(1+ \frac{N}{2} +\beta ( \frac{p-1+r}{p-1}) \right) \frac{-\gamma}{1+\gamma} -\frac{1}{2}} \right).
\end{eqnarray*}
Regarding the parameters $\theta_\infty$ and $\beta$ involved into  \eqref{assymp-formal-theta}, we derive the
 following system
 \begin{equation}\label{system-beta-theta-infty}
 \left\{       \begin{array}{rcl}
 \beta  & = &  \left(  1 + \frac{N}{2} + \beta \left( 1 +\frac{N}{2} \right)     \right)  \frac{\gamma}{\gamma +1},\\
\theta_\infty & = &  |\Omega |^{\gamma}  \left( \frac{(1+\gamma)  B}{1+ \frac{N}{2} +\beta ( \frac{p-1+r}{p-1})} \right)^{\frac{-\gamma}{\gamma+1}},
 \end{array}
 \right.
 \end{equation}
where $B$ defined as in \eqref{defi-B-constant}. We solve this system to derive the formulas of $ \beta$ and $ \theta_\infty$:
\begin{eqnarray}
\theta_\infty    &=&   \left(  \frac{ |\Omega|}{   \frac{r(1-\gamma \frac{N}{2})}{\frac{N}{2}+1}  \int_0^\infty \varphi_0^{p-1+r}(\xi)  \xi^{N-1} d\xi  } \right)^{\frac{\gamma}{1 - \frac{\gamma N}{2}}}, \label{defi-theta-infty} \\[0.2cm]
\beta  &=& \left(\frac{N}{2} +1 \right) \frac{\gamma  }{1  - \gamma \frac{N}{2}}.\label{defi-beta}
\end{eqnarray}
In particular, via Lemma \ref{bubles-integral}, we can reformulate  \eqref{defi-theta-infty} to 
\begin{equation}\label{new-formula-theta-infty}
\theta_\infty  =  \left(   \frac{|\Omega| \left( 1 + \frac{N}{2} \right) }{1  -  \frac{\gamma N}{2}} 2b^{\frac{N}{2}}    \right)^{\frac{\gamma}{1 - \frac{\gamma N}{2}}},
\end{equation}
which indeed imposes the following condition
 \begin{equation*}\label{condition-parameters}
\begin{array}{rcl}
1 -\frac{\gamma N}{2} > 0\implies \gamma<\frac{2}{N}.
\end{array} 
 \end{equation*}
However,  the above condition will be restricted drastically in the sequel, where we will need to assume that  $\gamma \le \gamma_0$, with $\gamma_0 $ small enough.

 \section{Formulation of the problem}\label{section-formulation}
 In this section,  we  formulate   the  problem   treated by  Theorem   \ref{theorem-existence}. We should point out that the rigorous approach differs at some points  from the formal one described in the previous section.
 
 \subsection{Similarity variable } Let us consider   $ u $ be  a  solution  to \eqref{equa-u-non-local},   and $ U$ is defined  as follows
 \begin{eqnarray}
 U(x,t) :=  \theta^{\frac{1}{p-1}}(t) \chi_1(x,t) u(x,t),\label{defi-U-rigorous}
 \end{eqnarray}
 where $\theta$ introduced as in \eqref{defi-theta-} and $\chi_1$ defined by
 \begin{equation}\label{c4defini-psi-M-0-cut}
  \chi_1 (x,t)  =   \chi_0 \left(     \frac{|x|}{K_0 \sqrt{T-t}|\ln(T-t)|}      \right),
\end{equation}
with 
 $\chi_0 \in C_0^\infty ([0,+\infty)) $, satisfying 
\begin{equation}\label{c4defini-chi-0}
supp (\chi_0) \subset [0,2], \quad 0 \leq \chi_0(x) \leq 1, \forall x \text{ and  }   \chi_0 (x)= 1, \forall x \in [0,1].
 \end{equation}
Note that the definition in  \eqref{defi-U-rigorous} is quite different from the one provided by \eqref{defi-U-by-u}  in the formal approach. Indeed the latter one gives rise to a solution blowing up at any point of $\Omega$, while our purpose is to construct a solution blowing up only at  the origin.  Due to  \eqref{equa-u-non-local}   $U$ solves
\begin{eqnarray}
\partial_t U  = \Delta U     + U^p  + \left( \frac{1}{p-1} \frac{\theta'(t)}{\theta(t)} -1 \right) U  +       F(u,U),\label{equa-w-rigorous}
\end{eqnarray}
where 
\begin{equation}\label{c4defini-F-1}
F(u,U) = \left\{   \begin{array}{rcl} & &
-u \theta^\frac{1}{p-1}(t) \Delta \chi_1 -2\theta^{\frac{1}{p-1}} \nabla \chi_1 \cdot \nabla u + (\theta^\frac{1}{p-1} u )^{p} (\chi_1-\chi_1^p) \\
& &  \text{ if }   |x| \le 2K_0 \sqrt{(t-t)|\ln(T-t)| } |\ln(T-t)|^\frac{1}{2},    \\
 & &  0,   \text{ otherwise}. 
\end{array}   \right. 
\end{equation}
Note that  $F(u,U)$ is  regarded as a small perturbation term. Now, using again \eqref{similarity-variable} we deduce that $W$ solves
\begin{eqnarray}
\partial_s W =  \Delta W - \frac{1}{2} y \cdot \nabla W - \frac{W}{p-1}     + W^p + \left( \frac{1}{p-1}\frac{\bar \theta_s(s)}{\bar \theta(s)}  - e^{-s} \right)  W  + \tilde F(y,s) ,\label{equa-W-rerogous}
\end{eqnarray}
where $\tilde F(y,s)$ defined by
\begin{eqnarray}
\tilde{ F}(y,s):= e^{-\frac{p}{p-1} s} F(u,U)(y,s).\label{defi-tilde-F}
\end{eqnarray}
Next we consider the   linearization around the profile, $\varphi$ given by \eqref{defied-varphi}:
\begin{equation}\label{defi-q=w-varphi}
    q: =W- \varphi,
\end{equation}
hence, $q$ solves 
 \begin{equation}\label{c4equa-Q}
\partial_s q   =  ( \mathcal{L} + V)  q    +  B(q) + R(y,s) + G(w,W),   
\end{equation}
where
\begin{eqnarray}
\mathcal{L}  & =&  \Delta  - \frac{1}{2} y  \cdot  \nabla  + Id,\label{c4defini-ope-mathcal-L}\\
V(y,s) & = &        p \left(  \varphi^{p-1}(y,s)    -  \frac{1}{p-1}     \right)  ,\label{c4defini-potential-V}  \\
B(q)  &    =&     \left(q +  \varphi  \right)^p   -   \varphi^p -  p \varphi^{p-1} q      , \label{c4defini-B-Q}\\
R(y,s) & = &  -   \partial_s\varphi +       \Delta \varphi - \frac{1}{2}  y \cdot \nabla \varphi - \frac{\varphi}{p-1}     + \varphi^p, \label{c4defini-rest-term}\\
G(.,s)  & =&  \left( \frac{1}{p-1} \left[ \frac{ \bar \theta' (s)}{ \bar \theta (s)} \right]  - e^{-s} \right) \left(  q+ \varphi \right) + \tilde F. \label{c4defini-N-term} 
\end{eqnarray}

\medskip
 In the following we recall some properties of the linear operator $\mathcal{L} $ and the potential $ V$. 

\begin{center}
   \textbf{Operator $\mathcal{L}$}
\end{center}
Let us point out that  operator $\mathcal{L}$ is  exactly  the same as in  \cite{DNZarxiv2020}, \cite{DZM3AS19} and  the references therein.  The interested readers can find more details about $\mathcal{L}$ in those works while here we only present its main  properties. Indeed, $\mathcal{L}$  is  self-adjoint in  $\mathcal{D} (\mathcal{L}) \subset L^2_\rho (\mathbb{R}^N),$  where  
$$   L^2_\rho  (\mathbb{R}^N) = \left\{        f \in L^2_{loc} (\mathbb{R}^N) \left| \right. \int_{\mathbb{R}^N}  |f(y)|^2 \rho (y) dy  < + \infty   \right\},$$
and
$$ \rho (y):  =  \frac{e^{-\frac{|y|^2}{4}}}{ (4\pi )^{\frac{N}{2}}}. $$
The  spectrum set is explicitly given by
$$   \text{Spec} (\mathcal{L}) =  \left\{    1 -  \frac{m}{2}    \left|  \right.      m \in \mathbb{N}\right\}   .$$
The eigenspace corresponding to  $\lambda_m = 1  - \frac{m}{2}$ is  given  by 
\begin{equation}\label{eigen-space-E-j}
 \mathcal{E}_m =  \left<   h_{m_1} (y_1). h_{m_2} (y_2).... h_{m_N} (y_N)   \left|  \right.  m_1 + ...+ m_N = m    \right>, 
\end{equation}
where   $h_{m_i}$  is  the (rescaled) Hermite  polynomial in one  dimension. 

\medskip

\begin{center}
    \textbf{Potential  $V$}
\end{center}
$V$ satisfies the following 
\begin{itemize}
\item[$(i)$]   $V(., s) \to 0$  in $L^2_\rho(\mathbb{R}^N)$ as  $s \to + \infty$ and   it has  some perturbations    on     $\mathcal{L}$'s effect.
\item[$(ii)$]   $V(y,s)$ is almost a constant   outside the blowup region,  i.e for  $|y| \geq K_0 \sqrt s.$ In particular,    we have the following estimate
$$     \sup_{ s \geq  s_\epsilon, \frac{|y|}{ \sqrt s}  \geq \mathcal{C}_\epsilon} \left|  V(y,s)   - \left( -\frac{p}{p-1}\right)  \right|   \leq  \epsilon, $$
for some   $\epsilon > 0$,     $ \mathcal{C}_\epsilon > 0$, and  $s_\epsilon$.  We  also remark that   $ - \frac{p}{p-1 }  <   -1  .$ Therefore, since the  largest  eigenvalue  of  $\mathcal{L}$ is equal to   $1,$ then we have that  $\mathcal{L} + V$ has  a  strictly  negative  spectrum in the region $\{ |y| \geq K_0 \sqrt{s} \}.$ Thus,  we can easily  control   the solution in that  region  when $K_0$ is large enough.
 \end{itemize}

\medskip 
 As it is evident form the above,    $\mathcal{L} + V$   has not  the  same  behavior  inside  and outside    the  singular  domain $\{ |y| \leq   K_0 \sqrt{s}\}$.  So,  we will use a classical decomposition introduced in \cite{BKnon94} (see also \cite{MZnon97, MZdm97, DZM3AS19, DNZarxiv2020}). In particular for each     $r  \in L^\infty(\mathbb{R}^N)$, we write
\begin{equation}\label{c4R=R-b+R-e}
r (y)=   r_b (y) +  r_e (y)  \equiv  \chi (y,s)   r(y) + (1- \chi (y,s) )    r (y),
\end{equation} 
where   $\chi (y,s)$   defined by
\begin{equation}\label{c4defini-chi-y-s}
\chi (y,s)  = \chi_0 \left(  \frac{|y|}{K_0 \sqrt s}   \right),
\end{equation}
and  $\chi_0 $ is as introduced by \eqref{c4defini-chi-0}.  Let us remark that
\begin{eqnarray*}
\text{Supp }(r_b)  & \subset & \{ |y| \leq 2 K_0 \sqrt{s} \},\\
\text{Supp }(r_e) & \subset & \{ |y| \geq K_0 \sqrt{s}\}.
 \end{eqnarray*}
In addition,  for all   $r_b \in L^2_\rho \left( \mathbb{R}^N \right)$  we  write 
\begin{eqnarray*}
r_b (y)  & =  &r_0  + r_1 \cdot y +   y^T \cdot r_2 \cdot y  - 2 \text{ Tr}(r_2)  + r_-(y),
\end{eqnarray*}
or alternatively
\begin{eqnarray*}
r_b (y)  & =  &r_0  + r_1 \cdot y +r_\perp(y) ,
\end{eqnarray*}
where  
\begin{equation}\label{c4defini-R-i}
r_i =   \left(   P_\beta ( r_b )  \right)_{\beta \in \mathbb{N}^N, |\beta|= i}, \forall  i \geq 0,
\end{equation}
with $P_\beta(r_b)$  being  the projection  of  $r_b$ on   the   eigenfunction    $h_\beta$ defined as follows:
\begin{equation}\label{c4defin-P-i}
P_\beta (r_b) =  \int_{\mathbb{R}^N}  r_b   \frac{h_\beta}{\|h_\beta\|_{L^2_\rho(\mathbb{R}^N)}^{-2}}   \rho dy, \forall \beta \in \mathbb{N}^N,
\end{equation}
and 
\begin{equation}\label{c4defini-R-perp}
r_\perp =  P_\perp (r)  =   \sum_{\beta \in \mathbb{N}^N, |\beta| \geq 2}   h_\beta P_\beta (r_b).
\end{equation}
Furthermore, $r_-$ is defined by
\begin{equation}\label{c4defini-R--}
r_-   =   \sum_{\beta \in \mathbb{R}^N, |\beta| \geq 3}   h_\beta P_\beta (r_b).
\end{equation}
 Finally,   we will use the two following expansions    
\begin{eqnarray}
r (y) &  =  &  r_0 + r_1 \cdot y  +  y^T \cdot r_2 \cdot y  - 2 \text{ Tr}( r_2)   + r_- + r_e (y).\label{c4represent-non-perp}
\end{eqnarray}
or
\begin{eqnarray}
r (y) &  =  &  r_0 + r_1 \cdot y  + r_\perp(y) + r_e (y).\label{c4represent-non-perp1}
\end{eqnarray}
Our final goal is to construct a global solution $q$ on $[s_0,+\infty) $ such that
\begin{equation}\label{purpose-L-infty-q}
\| q(s)\|_{L^\infty} \to 0 \text{ as } s \to +\infty.
\end{equation}

\subsection{ Localization  variable} In this part,  we  aim to  formulate  the equation of the rescaled solution $\mathcal{U}$ corresponding to  $u$    in region $P_2$. Generally, the rigorous problem is similar to the one given at page  \pageref{formal-rescaled-U-P-2}. However, we must to explain   some details that help the readers to understand our  goal. First, let us consider $ u $ which exists  on $[0,t_1]$ for some $t_1 \in (0,T)$. Using equation \eqref{equa-u-non-local} and  $\mathcal{U}$'s definition given in  \eqref{c4rescaled-function-U}, we derive 
\begin{equation}\label{equa-rescaled-rigogous}
    \partial_\tau \mathcal{U} =  \Delta_\xi \mathcal{U}   + \left( \tilde \theta(\tau(x,t)) (\theta(t(x)) )^{-1}  \right) \mathcal{U}^p       - \rho (x)\mathcal{U}, \tau \in \left(-\frac{t(x)}{T-t(x)}, \frac{t_1-t(x)}{T-t(x)} \right),
\end{equation}
where $\theta$ defined as in  \eqref{defi-theta-}. In particular,  once $x \in P_2(0)$, we immediately have
$t(x) \le 0 $, then we should understand 
\begin{equation}\label{extend-t-x-le-0}
    \theta(t(x)) = \theta(0),
\end{equation}
and $\tilde{\theta}$ be  assumed similarly by 
\begin{equation}\label{defi-tilde-theta-tau}
   \tilde{\theta}(\tau') = \left\{   \begin{array}{rcl}
 & &  \theta(\tau' \rho(x)  +t(x) ) \text{ if  }  \tau' \rho(x) +t(x) \ge 0,\\
& &  \theta(0)  \text{ if  }   \tau' \rho(x)  +t(x) \le 0.
   \end{array}
\right.
\end{equation}
The main goal is  the following: for all $t\in [0,t_1)$ and $x \in P_2(t)$, we have
$$ \left| \mathcal{U}(x,\xi, \tau(x,t)) - \hat{\mathcal{U}}(x, \tau(x,t)) \right| \le \delta_0,  $$
where $ \delta_0$ will be small, and $\tau(x,t) = \frac{t-t(x)}{T-t(x)} = \frac{t-t(x)}{\rho(x)}$ and 
\begin{equation}\label{defi-mathcal-U-rigorous}
\hat{\mathcal{U}}(x,\tau) =  \left(  (p-1)  \left[ 1 -  \int_0^\tau \tilde \theta (\tau') (\theta^{-1}(t(x))) d\tau'  \right]    + b\frac{K_0^2}{16} \right)^{-\frac{1}{p-1}}.
\end{equation}
 Note that  \eqref{defi-mathcal-U-rigorous}  makes sense in the following region
\begin{eqnarray*}
x \in  \left[ \frac{K_0}{4} \sqrt{(T-t)|\ln(T-t)|}, \epsilon_0 \right] ,  \tau' \in \left[ 0,\tau(x,t) \right],
\end{eqnarray*}
since $\tilde \theta (\tau') (\theta(t(x)))^{-1} \le 1,  \forall \tau' \in [0,\tau],$  with  $\tau \in \left[\frac{-t(x)}{\rho(x)}, \frac{t_1-t(x)}{\rho(x)} \right]$, and  $\rho$ defined as in \eqref{c4defini-theta}. This yields 
$$ 1 -  \int_0^\tau \tilde \theta (\tau') \left( \theta(t(x)) \right)^{-1}  d\tau'   \ge 1-\tau \ge 0.$$
Hence, we get the fact that $\hat{\mathcal{U}}(x,\tau)$  is well defined. In particular, such a $\hat{\mathcal{U}}(\tau)$ solves  the  following system:
 \begin{equation}\label{equa-hat-mathcal-U}
 \left\{  \begin{array}{rcl}
 \partial_\tau  \hat{\mathcal{U}} (\tau)   & = & \tilde \theta(\tau) \left(\theta(t(x)) \right)^{-1}  \hat{ \mathcal{U}}^p (\tau),\\[0.2cm]
 \hat{\mathcal{U}} (0)  & = &  \left(  p-1  + \displaystyle b\displaystyle \frac{K_0^2}{16} \right)^{-\frac{1}{p-1}}.
\end{array} \right.
 \end{equation}

\section{The existence proof without technical details}\label{estistence-without-technical} 

The main goal of the current section is to construct  $q$, the solution to \eqref{c4equa-Q}, satisfying the asymptotic behavior \eqref{purpose-L-infty-q}.    For readers' convenience, we aim to present  the proof of Theorem \ref{theorem-existence} at subsection \ref{proo-theorem-existence}  as well as other related complementary results in without technical details.

\subsection{Shrinking set }
In the sequel, we  focus on  building a   special set  where   
the  solution's behavior is controlled properly, leading to the derivation of the asymptotic behavior \eqref{purpose-L-infty-q}.  The construction of this set  is inspired by    \cite{DNZarxiv2020}, whilst  some required modifications should be implemented  for the underlying  critical regime case.
\begin{definition}[Shrinking set]\label{defini-shrinking-set-S-t} Consider  $T, K_0, \epsilon_0, \alpha_0, A, \delta_0, C_0, \eta_0$ and  take $t \in [0, T)$ for some  $T>0.$  We define the following  set    
$$ S( T, K_0 ,\epsilon_0, \alpha_0, A, \delta_0, C_0, \eta_0, t) \quad   (S(t) \text{ in short}),$$
as a subset  of $C^{2}\left( \Omega    \right) \cap C(   \bar \Omega ),$  containing  all functions  $u$  satisfying the following conditions:
\begin{itemize}
\item[$(i)$] \textbf{Estimates  in  $P_1(t)$}: In that region the function  $q$ introduced   in  \eqref{defi-q=w-varphi} belongs to $V_{ A} (s)\subset L^\infty (\mathbb{R}^N),  s = -  \ln(T-t),$ with   each  $r\in V_{ A} (s)$  satisfying the following estimates:
\begin{eqnarray*}
 &&| r_0   |    \leq   \frac{A^3}{s^{\frac{3}{2}}}, \text{ and  }  |q_1| \le \frac{A}{s^2}   \text{ and  }  | r_2 | \leq  \frac{A^4 }{s^\frac{3}{2}},\\
&&|r_-(y)|  \leq    \frac{A^6}{s^2} (1 + |y|^3)   \text{ and }  \left|   \left(   \nabla r  \right)_\perp    \right|   \leq  \frac{A^6}{s^2} (1 + |y|^3),  \forall y \in \mathbb{R}^N,\\
&&\| r_e \|_{L^\infty(\mathbb{R}^N)} \leq    \frac{A^7}{ \sqrt s}, 
\end{eqnarray*}      
where   $r_i $, $ r_-, (\nabla r)_\perp$ and $r_e$  introduced in  \eqref{c4defini-R-i}, \eqref{c4defini-R-perp}, \eqref{c4defini-R--}, and  \eqref{c4R=R-b+R-e},  respectively.
\item[$(ii)$] \textbf{ Estimates in  $ P_2 (t)$:}     For all   $|x|    \in \left[  \frac{K_0}{4 }  \sqrt{(T-t)|\ln(T-t)|},  \epsilon_0 \right], \tau (x,t)  =   \frac{t- t(x)}{\varrho (x)}$    and $|\xi|  \leq  \alpha_0  \sqrt{|\ln \varrho(x)|},$   the following hold
\begin{eqnarray*}
\left|   \mathcal{U}  (x, \xi, \tau(x,t))  - \hat{\mathcal{U}}(x,\tau(x,t))  \right|  & \leq &   \delta_0, \\
\left|  \nabla_\xi  \mathcal{U}  (x, \xi, \tau(x,t))  \right|  & \leq  & \frac{C_0}{\sqrt{|\ln \varrho(x)|}},\\
\end{eqnarray*}
where   $\mathcal{U},$  $\hat{\mathcal{U}}$ and  $\varrho(x) $  defined as in \eqref{c4rescaled-function-U},  \eqref{c4defin-hat-mathcal-U-tau}  and  \eqref{c4defini-theta},   respectively.
\item[$(iii)$] \textbf{Estimates in  $P_3(t)$:}  For all  $  x \in  \{ |x| \geq \frac{\epsilon_0}{4} \} \cap  \Omega,$ we have 
\begin{eqnarray*}
\left|    u(x,t)  -  u(x,0)  \right|  & \leq  &  \eta_0,\\
\left|  \nabla u(x,t) -  \nabla e^{t \Delta} u(x, 0) \right| & \leq  &  \eta_0, 
\end{eqnarray*}
where  $e^{t\Delta}$ is the semi-group generated by $\Delta$  with Neumann boundary conditions.
\end{itemize}
\end{definition}
       By Definition  \ref{defini-shrinking-set-S-t}$(i)$,   we can estimate  $q$'s size  as follows
\begin{lemma}[Growth estimates] \label{lemma-properties-V-A-s}  We consider  $K_0 \geq 1$ and  $A \geq 1$. Then,  there exists $s_1=s_1(A,K_0)$ such that for all $ s \geq s_1$ and  $q\in V_{A}(s),$  the following hold:
$$   |q (y,s)| \leq  \frac{C (K_0)A^7}{\sqrt s }   \text{ and }    \left| q (y,s)\right| \leq   \frac{C (K_0) A^7 }{s^\frac{3}{2}} (1 +   |y|^3). $$
In particular,  we have
\begin{eqnarray*}
\|  q\|_{L^\infty}({\{|y| \leq K_0 \sqrt{s} \})}  \leq C(K_0) \frac{A^6 }{s^\frac{3}{2}} (1 + |y|^3) .
\end{eqnarray*}
\end{lemma}
\begin{proof}
The proof immediately follows from  Definition \ref{defini-shrinking-set-S-t}$(i)$ and so the details are omitted.
\end{proof}

\subsection{Constructing appropriate initial data } 
In  this  paragraph, we aim to  build initial data $u_0 \in S(0)$ for equation \eqref{equa-U-theta-'-theta}. Let us consider $\chi_0$ and $\chi_1$ defined as in   \eqref{c4defini-psi-M-0-cut} and \eqref{c4defini-chi-0}, respectively.  Next,  we  introduce  $H^*$ as a suitable modification of the final asymptotic profile in the intermediate region in Theorem \ref{theorem-existence}: 
\begin{equation}\label{c4defini-H-epsilon-0}
   H^*  (x)  =  \left\{   \begin{array}{rcl}
&  &  \left[  \frac{|x|^2}{|\ln|x||}    \right]^{ -\frac{1}{p-1}} 
 \times   \theta_\infty^{-\frac{1}{p-1}} (2|\ln|x||)^{\frac{\beta}{p-1}} , \quad  \forall   |x| \leq  \min\left(  \frac{1}{4} d (0 ,\partial \Omega), \frac{1}{2}  \right), x \neq 0, \\[0.7cm]
&  & 1,   \quad   \forall   |x| \geq  \frac{1}{2}  d(0, \partial \Omega),
\end{array}      \right. 
\end{equation}
where $\beta, \theta_\infty$ defined as in   \eqref{defi-beta}, \eqref{defi-theta-infty}, and $\tau_0(x) =   -\frac{t(x)}{T-t(x)} .$

\medskip
Taking  $(d_0, d_1 ) \in   \mathbb{R}^{1 + N}$, we now  define initial data as follows:
\begin{eqnarray}
u_{d_0, d_1}(x,0)   & = & T^{-\frac{1}{p-1}} \theta_\infty^{-\frac{1}{p-1}} |\ln T|^{-\frac{\beta}{p-1}}  \left[  \varphi \left( \frac{x}{\sqrt{T}}, - \ln s_0  \right)   +  \left(d_0 \frac{A^3}{s_0^\frac{3}{2}}+  \frac{A}{s^2_0} d_1  \cdot  \frac{x}{\sqrt{T}}  \right)  \chi_0 \left(  \frac{|z_0|}{\frac{K_0}{32}}\right) \right]  \chi_1 (x) \nonumber\\
 & + &    H^*(x) \left(  1 -  \chi_1 (x,0)    \right)  \label{c4defini-initial-data},
\end{eqnarray}
where   $ z_0 = \displaystyle  \frac{ x}{ \sqrt{T  |\ln T|}} , s_0 = -\ln T$;     $\varphi, \chi_0,   \chi_1 $ and $ H^*  $ are defined as in  \eqref{defied-varphi}, \eqref{c4defini-psi-M-0-cut}, \eqref{c4defini-chi-0} and  \eqref{c4defini-H-epsilon-0} respectively.  

\medskip
Note that we can also  write the initial data using the similarity variable given by \eqref{similarity-variable}, that is in terms of $y_0 = \frac{x}{\sqrt{T}} \in \Omega_{s_0}$ and $s_0 = - \ln T $.

- For $U_{d_0,d_1}(0):$  We can use  \eqref{c4defini-initial-data} to express $U_{d_0,d_1}(0)$ in terms of $\theta(0)$ defined as in  \eqref{defi-theta-}. Then,  via  \eqref{defi-U-by-u} it follows
\begin{equation}\label{defi-U-0}
U_{d_0,d_1}(0) =  \theta^{\frac{1}{p-1}}(0)\chi_1(x,0) u_{d_0,d_1}(0). 
\end{equation}

- For $W_{d_0,d_1}(y_0,s_0)$:  We consider $U_{d_0,d_1}(0)$ as in \eqref{defi-U-0} and use \eqref{similarity-variable} to derive  $W_{d_0,d_1}(y_0,s_0)$ \label{defi-W-s-0}

- For  $q_{d_0,d_1}(y_0,s_0)$:  We rely on the following definition
\begin{eqnarray}\label{defi-q-y-0-s-0}
q_{d_0,d_1}(y_0,s_0)     =   w(y_0, s_0) - \varphi(y_0,s_0),  
\end{eqnarray}
where $\varphi$ defined as in  \eqref{defied-varphi}.  

In the following, we construct appropriate initial data of the form \eqref{c4defini-initial-data}, i.e. initial data which belong to the shrinking set $S(0).$
\begin{proposition}[Constructing initial data]\label{c4proposiiton-initial-data} There exists a positive constant $K_{2}> 0$ large enough such that  for all  $K_0 \geq K_{2} $ and $  \delta_{ 2 } > 0,$ there exist  $C_{2}(K_0) > 0$ and $\alpha_{2} (K_0, \delta_{3}) > 0 $  such that for all  $\alpha_0 \in   (0, \alpha_{2}]$,     there exists  $\epsilon_{2} (K_0, \delta_{2}, \alpha_0) > 0$ such that   for all  $\epsilon_0 \in (0,\epsilon_{2}]$  and  $A \geq 1$, we can find  $T_{2} (K_0,  \delta_{2}, \epsilon_0, A,C_2 ) > 0 $ small enough,  such that  for all  $T \leq    T_2$ and $ s_0 = | \ln T|,$   and  initial data $u_{d_0,d_1}(0)$ as in \eqref{c4defini-initial-data},  the following properties  hold:

\medskip
\textbf{(I)} For all   $|d_0|, |d_1|  \le 2$ the initial data $u_{d_0,d_1}(0)$  satisfy 
 the following   estimates:
 
\begin{itemize}
\item Estimates in  $P_1(0):$    $q_{d_0,d_1} (s_0) $ defined in \eqref{defi-q-y-0-s-0},  satisfy
$$ |q_0 (s_0)| \leq  \frac{A^3}{s_0^\frac{3}{2}}, \quad   |q_{1,j}  (s_0)| \leq \frac{A}{s_0^2}, \quad  |q_{2,i,j}(s_0)| \leq \frac{1}{s_0^2},   \forall i,j \in \{1,..., N\},$$
$$ |q_- (y,s_0) | \leq 	\frac{1}{s_0^2} (|y|^3 + 1), \quad  |\nabla q_\perp (y,s_0) | \leq \frac{1}{s_0^2} (|y|^3 + 1),  \forall y \in \mathbb{R}^N,$$
and
$$  \|q_e   \|_{L^\infty}  \leq \frac{1}{\sqrt{s_0}}.$$
\item Estimates  in  $P_2(0)$: For all $|x| \in \left[  \frac{K_0}{4} \sqrt{T |\ln T|}  , \epsilon_0 \right] , \tau_0(x) = - \frac{ t(x) }{ \varrho (x)}$  and $|\xi|  \leq 2\alpha_0 \sqrt{|\ln \varrho(x)|},$ we have
$$   \left|  \mathcal{U} (x, \xi, \tau_0(x) )    - \hat{\mathcal{U}} (x,\tau_0(x))  \right| \leq   \delta_{3} \text{ and }    |\nabla_\xi \mathcal{U} (x, \xi, \tau_0(x))| \leq \frac{C_{3}}{\sqrt{|\ln \varrho(x)|}} ,
 $$
\end{itemize}
where   $\mathcal{U}, \hat{\mathcal{U}}, $ and  $\varrho(x)$ are  defined  as in  \eqref{c4rescaled-function-U}, \eqref{c4defin-hat-mathcal-U-tau} and  \eqref{c4defini-theta}, respectively.
\textbf{(II)}  There exits   $\mathcal{D}_{ A}     \subset [-2,2] \times [-2,2]^N$  such that      the following mapping
\begin{eqnarray*}
\Gamma :  \mathbb{R}^{1+N}   & \to &   \mathbb{R}^{1 +N}\\
(d_0, d_1)   & \mapsto   &  \left( q_0, q_1  \right)(s_0),
\end{eqnarray*}
   is  affine,  one to  one   from   $\mathcal{D}_{A}$   to 	$\hat{\mathcal{V}}_A (s_0),$ where $\hat{\mathcal{V}}_A (s)$  defined by
\begin{equation}\label{c4defini-hat-mathcal-V-A}
\hat{\mathcal{V}}_A (s) =  \left[-\frac{A^3}{s^\frac{3}{2}}, \frac{A^3}{s^\frac{3}{2}} \right] \times \left[  -\frac{A}{s^2}, \frac{A}{s^2} \right]^N.
\end{equation}  
In addition  we have   
$$\Gamma \left|_{\partial \mathcal{D}_{A}}  \right.   \subset   \partial \hat{\mathcal{V}}_A (s_0), $$
and 
\begin{equation}\label{c4deg-Gamma-1-neq-0}
 \text{deg} \left(  \Gamma \left. \right|_{ \partial  \mathcal{D}_{A} }  \right) \neq 0,
\end{equation}
  where   $q_0,q_1$  considered  as $q_{d_0,d_1} (s_0)$'s  components and  $q_{d_0,d_1}(s_0)$  defined as in   \eqref{defi-q-y-0-s-0}.

 \end{proposition}
\begin{proof}
Notably,  the shrinking set and initial data are  the same as in \cite[Lemma 2.4]{MZnon97}, therefore  the proof   is basically based on  that Lemma. However, the situation in the current work is more complicated due to the  presence of non-local term $\theta(t).$  For reader's convenience, we   aim  to provide a complete proof at  Section \ref{proof-initial-data-propo}.  
\end{proof}

\subsection{Contribution on the non-local term}

In this  section, we will show the asymptotic behavior  of  the   non-local term $\theta(t)$  defined by \eqref{defi-theta-}  which is   Proposition  \ref{propo-bar-mu-bounded} in the below.

\medskip
First,  we aim to give some    estimates on $u$ once it is trapped in shrinking set $S(t)$. 
\begin{lemma}\label{lemma-estimate-U-x-t-in-S-t}
Let us consider $u \in S(K_0, \epsilon_0, \alpha_0, A, \delta_0, C_0, \eta_0, t)$ for all $t \in [0,T)$, defined as  in Definition \ref{defini-shrinking-set-S-t}  where  $ \eta_0 \ll 1$.   Then, the following hold:
\begin{itemize}
    \item[$(i)$] For all $|x| \leq K_0 \sqrt{(T-t) } |\ln(T-t)|$, we have 
    \begin{eqnarray}
    \left| u(x,t ) - \theta^{-\frac{1}{p-1}}(t)(T-t)^{-\frac{1}{p-1}} \varphi_0\left( \frac{|x|}{\sqrt{(T-t)|\ln(T-t)|}}\right) \right|\leq \frac{CA^7(T-t)^{-\frac{1}{p-1}}  }{1 + \sqrt{|\ln(T-t)|}},\label{estimate-u-varphi-0}
    \end{eqnarray}
 where $\varphi_0$ defined by \eqref{defi-varphi-0}, together with the gradient estimate
  \begin{eqnarray*}
    \left| \nabla_x u(x,t )  \right|\leq \frac{C(K_0)A^7(T-t)^{-\frac{1}{p-1}-\frac{1}{2}} \theta^{-\frac{1}{p-1}}(t)}{1 + \sqrt{|\ln(T-t)|}},
    \end{eqnarray*}
    In particular,   if  $  p \ge 3 $ and $\frac{r}{p-1} = \frac{N}{2}$, then we obtain
    \begin{eqnarray}
& &    \left|  u^{p-1+r}(x,t) -   \theta^{-\frac{p-1+r}{p-1}}(t) (T-t)^{-\frac{p-1+r}{p-1}}  \varphi_0^{p-1+r} \left( \frac{|x|}{\sqrt{(T-t)|\ln(T-t)|}}\right)  \right|\label{estima-u-p-1+r}  \\
   &\leq & C A^{7(p-1)(1+\frac{N}{2})} (T-t)^{-\frac{p-1+r}{p-1}}        \theta^{-\frac{p-1+r}{p-1}}(t) |\ln(T-t)|^{-\frac{p-1}{2} (1+\frac{N}{2})}\nonumber  \\
   & + & CA^7   (T-t)^{ -\frac{p-1+r}{p-1} } \theta^{-\frac{p-1+r}{p-1}} (t) |\ln(T-t)|^{-\frac{1}{2}} \varphi_{0}^{p-2+r}  \left( \frac{|x|}{\sqrt{(T-t)|\ln(T-t)|}}\right) \nonumber.
    \end{eqnarray}
    \item[$(ii)$] For all $|x| \in \left[ \frac{K_0}{4} \sqrt{(T-t) |\ln(T-t)|}, \epsilon_0 \right]$, we have
    \begin{eqnarray*}
  \frac{1}{C} \left[ |x|^2 \right]^{-\frac{1}{p-1}} |\ln|x||^{\frac{1+\beta}{p-1}}  \le  |  u(x,t)| \leq C  \left[ |x|^2 \right]^{-\frac{1}{p-1}} |\ln|x||^{\frac{1+\beta}{p-1}},
    \end{eqnarray*}
    and
    \begin{eqnarray*}
    | \nabla_x u(x,t) | \leq  C(C_0) (|x|^2)^{-\frac{1}{p-1} -\frac{1}{2}}  |\ln|x||^{\frac{1+\beta}{p-1} -\frac{1}{2}},
    \end{eqnarray*}
    provided that $K_0 \geq K_6$ and $ \epsilon_0 \leq \epsilon_6(K_0)$. 
    \item[$(iii)$] For all $|x| \geq \epsilon_0$, we have
    $$ \frac{1}{2} \leq  u(x,t)  \leq C(\epsilon_0,\eta_0),$$
    and
    $$ | \nabla_x u(x,t)| \leq C(\eta_0, \epsilon_0),$$
    provided that $\eta_0 \ll 1$.
\end{itemize}
\end{lemma}
\begin{proof} The proof mainly bases on estimations provided in  Definition \ref{defini-shrinking-set-S-t} and it is quite the same  as [Lemma 6.2, \cite{DNZarxiv2020}]. Just to point out that in the current situation    $ K_0 \sqrt{(T-t)|\ln(T-t)|}$ in [Lemma 6.2, \cite{DNZarxiv2020}] is now replaced by   $ K_0\sqrt{(T-t)} |\ln(T-t)|$, however the same technique applies and \eqref{estima-u-p-1+r} follows \eqref{estimate-u-varphi-0}  in implementing   the fundamental inequality     $   | (a+b)^\alpha -b^\alpha| \le   C(\alpha)   (   ba^{\alpha-1}   +  b^\alpha     )      $ with $\alpha = p-1+r$.
\end{proof}
Next, we aim to give the  rigorous proof to   \eqref{assume-bar-theta-'-theta} assumed at the formal approach part:
\begin{proposition}[Dynamics of $\theta$]\label{propo-bar-mu-bounded} Let us consider  \eqref{critical-condition} and $p\ge 3$ with  $\Omega$ being a  bounded domain with smooth boundary.  Then, there  exists $K_{3}> 0$ such that  for all  $K_0 \geq K_{3},  \delta_{0 } > 0,$ there exists  $\alpha_{3} (K_0, \delta_{0}) > 0$  such that for all  $\alpha_0 \leq   \alpha_{3}$  we can find     $\epsilon_{3} (K_0, \delta_{0}, \alpha_0) > 0$ such that   for all   $\epsilon_0  \leq \epsilon_3$  and  $A \geq 1, C_0 > 0, \eta_0 > 0$, there exists $T_{3}   > 0$ such that  for  all  $T  \leq T_3$ the following holds:  Assuming   $U$ is a non negative  solution  of   equation \eqref{equa-U-theta-'-theta}    on $[0, t_1],$ for some $t_1 < T$ and  $U \in S(T,K_0, \epsilon_0, \alpha_0, A, \delta_0, C_0,\eta_0,t)= S(t) $ for all  $t \in [0, t_1]$. Then, the following hold
 \begin{eqnarray}
\left| \theta(t) -  \theta_\infty |\ln(T-t)|^{-\beta}  \right|  \le C  A^7  |\ln(T-t)|^{-\beta -\frac{1}{2}}, \label{bound-bar-theta}
\end{eqnarray}
and 
\begin{eqnarray}
\left| \theta'(t)  -  \theta_\infty(-\beta) (T-t)^{-1}|\ln(T-t)|^{-\beta-1}  \right| \le   \gamma C A^7(T-t)^{-1} |\ln(T-t)|^{-\beta -\frac{3}{2}},  \label{bound-derivative-bar-theta}
\end{eqnarray}
where $ \theta_\infty $ and $ \beta$ defined as in  \eqref{defi-theta-infty} and   \eqref{defi-beta} respectively.
In particular, if  we define   $\bar \theta(s) $ as in  \eqref{defi-bar-theta-s} and  take $ \gamma \le   A^{-4}$, then we obtain the following  estimate
\begin{equation}\label{esti-rigorous-theta-bar-}
\left|  \frac{\bar{\theta}'_s}{\bar \theta}  +   \frac{\beta}{s} \right|   \le \frac{CA^3}{s^\frac{3}{2}}.
\end{equation}
\end{proposition}
  \begin{proof}
   Let us assume that   the hypothesis in Lemma \ref{lemma-estimate-U-x-t-in-S-t} holds.   We point out that  \eqref{bound-bar-theta}   is similarly  processed as in the formal approach at  Subsection  \ref{subsection-justification-theta}   by virtue of    \ref{lemma-estimate-U-x-t-in-S-t}.

- The proof of    \eqref{bound-bar-theta}:  We recall that by \eqref{defi-theta-} we have
\begin{eqnarray}
\theta(t) =  |\Omega|^\gamma \left( \|u(t)\|^{r}_{L^r(\Omega)} \right)^{-\gamma}.\label{recall-defi-theta}
\end{eqnarray}
Next we estimate  $\| u\|^r_{L^r}.$ Indeed,  we first decompose the related integral as follows
\begin{eqnarray*}
\|u\|^r_{L^r} = \int_{|x| \le K_0 \sqrt{ (T-t)|\ln(T-t)| }} u^r + \int_{  K_0 \sqrt{ (T-t)|\ln(T-t)| } \le |x| \le \epsilon_0} u^r + \int_{|x| \ge \epsilon_0, x \in \Omega} u^r. 
\end{eqnarray*}
Using Lemma \ref{lemma-estimate-U-x-t-in-S-t} and  \eqref{recall-defi-theta}, we derive that
\begin{eqnarray*}
 \theta(t) \left( \theta^{-\frac{N}{2}}|\ln(T-t)|^\frac{N}{2} +  |\ln(T-t)|^{1+\frac{N}{2}(1+\beta)}  \right)^\gamma \lesssim 1.
\end{eqnarray*}
This implies that
\begin{eqnarray}
\theta(t) \le C |\ln(T-t)|^{-\gamma \left(1+\frac{N}{2}(1+\beta) \right)}  = C |\ln(T-t)|^\beta,\label{estimate-rough-theta-t}
\end{eqnarray}
since $\beta$ satisfies \eqref{system-beta-theta-infty}.

We also have
\begin{eqnarray*}
\theta'(t) =  \partial_t \left(  |\Omega|^\gamma ( \| u \|^r_{L^r(\Omega) } )^{-\gamma}  \right) =  |\Omega|^{\gamma} ( -\gamma)( \| u\|^r_{L^r(\Omega)})^{-\gamma-1} \partial_t ( \| u\|^r_{L^r(\Omega)}),
\end{eqnarray*}
and using again  \eqref{defi-theta-}  we obtain
$$   \frac{\theta'}{\theta^{1 + \frac{1}{\gamma}}}  = |\Omega|^{-1} (-\gamma)   \partial_t ( \| u\|^r_{L^r(\Omega)}). $$
For     $ \partial_t ( \| u\|^r_{L^r(\Omega)})$'s asymptotic, we would like to prove the following
\begin{eqnarray}
\partial_t ( \| u\|^r_{L^r(\Omega)})  &=& r \theta^{-1-\frac{r}{p-1}}(t) \left(  \int_{0}^\infty  \varphi_0^{p-1+r} (k) k^{N-1} dk  \right) (T-t)^{-1} |\ln(T-t)|^\frac{N}{2}  \nonumber   \nonumber\\
 &+&  O\left(  \theta^{-1-\frac{r}{p-1}} (t) A^7 (T-t)^{-1} |\ln(T-t)|^{\frac{N}{2}-\frac{1}{2}}  \right),\quad\mbox{as}\quad t\to T. \label{asym-partial-t-u-r-L-r}
\end{eqnarray}
Since the computation follows the same steps as in  Subsection \ref{subsection-justification-theta} by using now Lemma \ref{lemma-estimate-U-x-t-in-S-t} and \eqref{estimate-rough-theta-t}, the details are omitted. Note that  $ \frac{r}{p-1} = \frac{N}{2}$ then \eqref{asym-partial-t-u-r-L-r} takes the form
\begin{eqnarray*}
 \partial_t ( \| u\|^r_{L^r(\Omega)})  &=& r \theta^{-1 -\frac{N}{2}}(t) \left(  \int_{0}^\infty  \varphi_0^{p-1+r} (k) k^{N-1} dk  \right) (T-t)^{-1} |\ln(T-t)|^\frac{N}{2}  \nonumber   \\
 &+&  O\left(  \theta^{-1-\frac{N}{2}} (t) A^7 (T-t)^{-1} |\ln(T-t)|^{\frac{N}{2}-\frac{1}{2}}  \right) ,\quad\mbox{as}\quad t\to T. 
\end{eqnarray*}
Hence, we have
\begin{eqnarray}
\left( \theta^{ - \left(\frac{1}{\gamma} - \frac{N}{2} \right)} \right)' &=& -  \left(\frac{1}{\gamma} - \frac{N}{2} \right) |\Omega|^{-1} (-\gamma) r \int_0^\infty \varphi_0^{p-1+r}  (k) k^{N-1} dk  (T-t)^{-1} |\ln(T-t)|^\frac{N}{2} \nonumber \\
& + & O \left( A^7 \gamma (T-t)^{-1} |\ln(T-t)|^{\frac{N}{2} -\frac{1}{2}}\right),\quad\mbox{as}\quad t\to T.\label{derivative-theta-1-gama-N-2}
\end{eqnarray}
which yields 
\begin{eqnarray*}
\theta (t) &=&  \left(  \frac{r ( 1 - \frac{\gamma N}{2} ) \int_0^\infty \varphi_0(k) k^{N-1} dk}{ |\Omega| (\frac{N}{2} +1)}    \right)^{ -\frac{\gamma}{1 - \frac{\gamma N}{2}}} |\ln(T-t)|^{ -\frac{\gamma}{1 - \frac{\gamma N}{2}} } + O\left( A^7 |\ln(T-t)|^{ \frac{\gamma}{1 - \frac{\gamma N}{2}} -\frac{1}{2}}\right)\\
&=& \theta_\infty |\ln(T-t)|^{ -\beta} + O(A^7|\ln(T-t)|^{-\beta -\frac{1}{2}}), \quad\mbox{as}\quad t\to T.
\end{eqnarray*}
Thus, we conclude \eqref{bound-bar-theta}.

- The proof of \eqref{bound-derivative-bar-theta}   follows from \eqref{bound-bar-theta}  and  \eqref{derivative-theta-1-gama-N-2}. Finally,  for proving \eqref{esti-rigorous-theta-bar-} we  use   $\bar \theta$'s definition  together with \eqref{bound-derivative-bar-theta} and   \eqref{bound-bar-theta}.\end{proof}

\subsection{Proof of  Theorem  \ref{theorem-existence}}\label{proo-theorem-existence}
In this section, we provide  the proof of our main result Theorem \ref{theorem-existence}  skipping  many technicalities. All the missing technical details  will be given in a separate section later on.  
\begin{proposition}[Existence of  a solution belonging to S(t)]\label{c4proposition-existence-U}  
Let $\Omega $ be a smooth and bounded domain in $\R^N$ and  also assume that condition \eqref{critical-condition} is also satisfied.  Then,        there exist  positive parameters  $T, K_0, \epsilon_0,$ $\alpha_0, A, \delta_0, C_0,  \eta_0$  such that   for all $\gamma \in (0,A^{-5})$   there exists  a pair  $(d_0, d_1) \in \mathbb{R}^{1 +N}$ such that    equation \eqref{equa-U-theta-'-theta}   with   initial data $u_{d_0, d_1}(0)$, given  in \eqref{c4defini-initial-data},   has   a    unique solution   on  
  $ [0,T)$ and 
 $ u(t) \in S(t),  \text{ for all }  t \in [0,T)$,
 recalling that $S(t)$ is the shrinking set introduced in Definition \ref{defini-shrinking-set-S-t}.
\end{proposition}
\begin{proof}
Notably Theorem \ref{theorem-existence} arises easily by this proposition. Although, the current proposition is    relatively similar to    the  robustness argument,  given  in \cite{BKnon94, MZdm97, MZnon97, DNZarxiv2020},  some extra difficulties  arising in that case due to the new behavior of the non-local term $\theta.$     Hence, this result adds its own value to the  technique of construction of blowup solution in general.    For reader's convenience, we are completing complete proof in Section \ref{section-reduction-finite}.
\end{proof}

\begin{center}
    \textbf{Conclusion of Theorem \ref{theorem-existence}}
\end{center}
Let us fix positive parameters  $T, K_0, \epsilon_0, \alpha_0, A, \delta_0, C_0$ and $\eta_0$ such that Propositions \ref{propo-bar-mu-bounded}, and  \ref{c4proposition-existence-U} hold true. Then, we obtain 
$$ u(t) \in S(t) \quad  \forall t \in [0,T).$$
Using \eqref{bound-bar-theta} in  Proposition \ref{propo-bar-mu-bounded}, we derive
$$ \theta(t) = \theta_\infty |\ln(T-t)|^{-\beta} + O( |\ln(T-t)|^{-\beta -\frac{1}{2}}),\text{ as }  t \to T.$$
Next using Definition \ref{defini-shrinking-set-S-t} $(i)$, we directly deduce Theorem \ref{theorem-existence} $(i)$ and in particular estimate \eqref{asymptotic-theta-mai-novelty}. The latter estimate infers that $u$ blows up in finite time $T$  at the origin and  with speed
$$ u(0,t)  =  ( \theta_\infty)^{-\frac{1}{p-1}}    (T-t)^{-\frac{1}{p-1}} \left|\ln(T-t) \right|^{\nu }   \left(1 + O_{t \to T}\left( \frac{1}{\sqrt{|\ln(T-t)|}}  \right)\right), $$
recalling that $\nu = \frac{\beta}{p-1}.$

Besides, from Lemma  \ref{lemma-estimate-U-x-t-in-S-t}$(ii)$,   we deduce that  $u$  does not blow up
 at  any $ x \in \Omega \backslash \{ 0\},$ a single-point occurs.

 \medskip
 We now proceed with the proof for item $(ii)$:  In fact,  the existence of the blowup profile $u^*$   is quite the same as in \cite[Proposition 3.5]{DZM3AS19} and thus we only give a prove \eqref{final-profile}.  Consider  $\mathcal{U} $ defined as in  \eqref{c4rescaled-function-U} and $x\in (0, \epsilon_0)$ with $ \epsilon_0$ small enough. We  now   apply  \eqref{intermediate-profile} with $t =t(x)$ to obtain
\begin{equation}\label{math-final-tau=0}
\sup_{|\xi| \le 6 |\ln \rho (x)|^\frac{1}{4}}  \left| \mathcal{U}(x,\xi, 0)  -  \hat{\mathcal{U} }(x,0)   \right|  \le \frac{C}{ 1 + | \ln \rho (x)|^\frac{1}{4}},
\end{equation} 
 where    $ \hat{\mathcal{U} }(x,\tau) $   defined as in \eqref{defi-mathcal-U-rigorous}. By using the argument developed in [\cite{DZM3AS19}, Proposition 3.5, pages 1307-1310], we obtain
 \begin{equation}\label{esti-mahcau-U-final-pro}
 \sup_{|\xi| \le  |\ln \rho (x)|^\frac{1}{4},\tau \in [0,1)}  \left| \mathcal{U}(x,\xi, \tau)  -  \hat{\mathcal{U} }(x,\tau)   \right|  \le \frac{C}{ 1 + | \ln \rho (x)|^\frac{1}{4}},
 \end{equation}
 as well as 
 \begin{equation}\label{estimate-partial-tau}
  \sup_{|\xi| \le  |\ln \rho (x)|^\frac{1}{4},\tau \in [0,1)}  \left|\partial_\tau \mathcal{U}(x,\xi,\tau) \right| \le C(x).
 \end{equation}
 It follows that
 \begin{eqnarray*}
 u^*  & = & \lim_{\tau \to 1}   \theta_\infty^{-\frac{1}{p-1}}(T-t(x))^{-\frac{1}{p-1}} |\ln(T-t(x))|^\nu \mathcal{U}(x,0,\tau) \\
 & \sim &    \theta_\infty^{-\frac{1}{p-1}}(T-t(x))^{-\frac{1}{p-1}} |\ln(T-t(x))|^\nu \hat{\mathcal{U}} (x,1),
 \end{eqnarray*}
where  $\hat{U}(x,\tau)$ defined as in \eqref{defi-mathcal-hat-U-tau-0}. 
We now aim to prove that
\begin{equation}\label{prove-tildle-theta...-sim-1-more-preceise}
\int_0^1\tilde \theta(\tau') \theta^{-1}(t(x)) d\tau'  \sim 1 \text{ as } x \to 0,
\end{equation}
which is equivalent to  
\begin{equation}\label{prove-tildle-theta...-sim-1}
\int_0^1\tilde \theta(\tau') \theta_\infty^{-1} |\ln(T-t(x))|^\beta d\tau' \sim 1 \text{ as } x \to 0,
\end{equation}
since,  as $x \to 0$, we have $t(x) \to T$, and equality \eqref{bound-bar-theta}. Now, we focus on the proof of \eqref{prove-tildle-theta...-sim-1}.  Indeed,  using  $\tilde \theta $'s definition  given  in \eqref{defi-tilde-theta-tau} and the fact  \eqref{bound-bar-theta}  again, we then derive
 $$\left| \tilde \theta  (\tau) - \theta_\infty |\ln[(T-t(x))(1-\tau)]|^{-\beta} \right|\lesssim |\ln[(T-t(x))(1-\tau)]|^{-\beta -\frac{1}{2}}, \forall \tau \in [0,1).  $$
 Therefore it yields
 \begin{eqnarray*}
  \left| \theta(\tau) \theta_\infty^{-1} |\ln(T-t(x))|^\beta     - \frac{1}{  \left| 1 +  \frac{\ln(1-\tau)}{\ln(T-t(x))}  \right|^\beta}  \right| & \lesssim &   \frac{|\ln(T-t(x))|^\beta}{| \ln(T-t(x)) + \ln(1-\tau) |^{\beta +\frac{1}{2}}}\\
  & \lesssim & \frac{1}{|\ln(T-t(x)) +\ln(1-\tau)|^\frac{1}{2} \left| 1 + \frac{\ln(1-\tau)}{\ln(T-t(x)) }\right|^{\beta} }\\
  & \lesssim & \frac{1}{|\ln(T-t(x))|^\frac{1}{2} \left| 1 + \frac{\ln(1-\tau)}{\ln(T-t(x)) }\right|^{\beta} }, 
 \end{eqnarray*}
 since 
 $$ | \ln(T-t(x)) +  \ln(1-\tau)  | \ge |\ln(T-t(x))|.$$
We now decompose the integral as follows
\begin{eqnarray*}
\int_0^1\tilde \theta(\tau') \theta_\infty^{-1} |\ln(T-t(x))|^\beta d\tau' &=& \int_0^{1 -  e^{ - \sqrt{ |\ln(T-t(x))|}  }} \tilde{\theta}(\tau') \theta_\infty^{-1} |\ln(T-t(x))|^\beta d\tau' \\
& + & \int_{1 -  e^{ - \sqrt{ |\ln(T-t(x))|}  }}^1 \theta(\tau') \theta_\infty^{-1} |\ln(T-t(x))|^\beta d\tau' .
\end{eqnarray*}
As a matter of fact, the second integral can be bounded by
$$  \int_{1 -  e^{ - \sqrt{ |\ln(T-t(x))|}  }}^1 \tilde{\theta}(\tau') \theta_\infty^{-1} |\ln(T-t(x))|^\beta d\tau'  \le C e^{- \sqrt{|\ln(T-t(x))|}},$$
since 
$$|\tilde{\theta}(\tau) \theta_\infty^{-1} |\ln(T-t(x))|^\beta| \le C,\quad \forall \tau \in [0,1).$$
On the other hand,  if $\tau \in [0,1)$ such that
$$  0 \le  \frac{\ln(1-\tau)}{\ln(T-t(x))}  \le \frac{1}{|\ln(T-t(x))|^\frac{1}{2}},  $$
we equivalently have
$$ 0\ge \ln(1-\tau) \ge \frac{\ln(T-t(x))}{|\ln(T-t(x))|^\frac{1}{2}} = -|\ln(T-t(x))|^\frac{1}{2} ,$$
which  is equivalent again
\begin{eqnarray*}
1 \ge 1- \tau  \ge e^{-|\ln(T-t(x))|^\frac{1}{2}}  \Leftrightarrow  \tau \in \left[0, 1 - e^{-|\ln(T-t(x))|^\frac{1}{2}} \right].
\end{eqnarray*}
Then in  the above region, we can apply a Taylor expansion to derive
$$ \left| \tilde{\theta}(\tau) \theta_\infty^{-1} |\ln(T-t(x))|^\beta  - 1 \right|  \lesssim |\ln(1-\tau)| |\ln(T-t(x))|^{-\frac{1}{2}}, $$
which finally concludes  \eqref{prove-tildle-theta...-sim-1}. 

Furthermore,   using\eqref{prove-tildle-theta...-sim-1}, we  derive   
\begin{eqnarray*}
u^* (x)  \sim    \theta_\infty^{-\frac{1}{p-1}}(T-t(x))^{-\frac{1}{p-1}} |\ln(T-t(x))|^\nu \left( b \frac{K_0^2}{16} \right)^{-\frac{1}{p-1}}, \text{ as } x \to 0
\end{eqnarray*}
and thus by virtue of  Lemma \ref{lemma-t(x)} we obtain 

 \begin{eqnarray*}
u^* (x)  \sim    \theta_\infty^{-\frac{1}{p-1}} \left[  \frac{b}{2 } \frac{|x|^2}{|\ln|x||}\right]^{-\frac{1}{p-1}} \left| \ln|x| \right|^{\nu}, \text{ as } x \to 0,
\end{eqnarray*}
recalling that $\nu = \frac{\beta}{p-1},$ which actually proves statement $(iii)$ of Theorem \ref{theorem-existence}.

This concludes the proof of Theorem \ref{theorem-existence}.

\section{A finite dimensional   reduction}\label{section-reduction-finite}

The current section deals with the  reduction of the  problem of  controlling $u(t) \in S(t)$ to a finite dimensional  one  on controlling only the two positive spectrum modes  $q_0$ and $q_1$ in $\hat{\mathcal{V}}_A (s)$, introduced in  \eqref{c4defini-hat-mathcal-V-A}.  
\begin{proposition}[Reduction to a  finite  dimensional problem]\label{c4proposition-reduction-finite}  There exist positive parameters  $T,   K_0, \epsilon_0, \alpha_0, A, $ $\delta_0,  C_0$ and $ \eta_0$ such that  if we    assume that    $(d_0,d_1) \in \mathcal{D}_A$, defined as in \eqref{c4proposiiton-initial-data};   and with  initial data $u_{d_0,d_1}$,  constructed as in \eqref{c4defini-initial-data},    the solution  $u$ of equation \eqref{equa-u-non-local}  exists on $[0, t_1],$ for some $t_1 < T$.   Furthermore, if we assume that $ u \in   S (t)$ for all $ \forall t \in [0, t_1] \text{ and }  u(t_1) \in  \partial S (t_1)$ (see  $S(t)$'s definition   in Definition \ref{defini-shrinking-set-S-t}),  then, the following  properties  hold: 
	\begin{itemize}
	\item[$(i)$] At $s_1 =- \ln(T-t_1)$, we have   $(q_0, q_1) (s_1)  \in \partial \hat{\mathcal{V}}_A (s_1).$
	\item[$(ii)$]   We can find  $\nu_0 > 0$ such that 
	$$  (q_0, q_1) ( s_1  +  \nu) \notin    \hat{\mathcal{V}}_A (s_1 +  \nu), \forall \nu \in (0, \nu_0). $$ 
	In particular, we have the fact that   there  exists  $\nu_1 > 0$ such that 
	$$ u  \notin    S(t_1  +  \nu), \forall \nu \in (0, \nu_1).$$
	\end{itemize}
	\end{proposition}    
\begin{proof}
  The proof  directly   follows   from  \textit{a priori  estimates } in regions $P_1, P_2$ and $P_3$  of $S(t),$ introduced in Definition \ref{defini-shrinking-set-S-t}. Such \textit{a priori  estimates } are derived in Propositions \ref{c4propo-priori-P-1-later}, \ref{priori-estimate-P-2}  and \ref{priori-estimate-P-3} below.  The main reasoning  is basically  the same as in \cite{DNZarxiv2020,DZM3AS19}. However,  since the current situation is more difficult,  due to the form of  perturbation $\theta,$ some significant modification of the existing technique should be performed.
\end{proof}

\bigskip
Let us resume to main arguments   on   deriving   \textit{a priori estimates} in  Propositions \ref{c4propo-priori-P-1-later}, \ref{priori-estimate-P-2}  and \ref{priori-estimate-P-3}:

\medskip
- In  Proposition \ref{c4propo-priori-P-1-later}   we give estimates  such that the bounds in  Definition \ref{defini-shrinking-set-S-t}$(i)$ are satisfied. More precisely,  we show that except from the bounds on $q_0 $ and $q_1$,   the other estimates  are satisfied by stricter bounds.  

\medskip
- In Proposition  \ref{priori-estimate-P-2}  we  prove that  all  estimates in Definition \ref{defini-shrinking-set-S-t} $(ii)$ are   controlled   by stricter bounds.

\medskip
- Finally, in Proposition \ref{priori-estimate-P-3} we again obtain that the required bounds in  Definition \ref{defini-shrinking-set-S-t}$ (iii)$  are also    satisfied by   stricter bounds.

\subsection{A priori estimates on $P_1$} 

We first establish the following result:
\begin{lemma}\label{lemma-priori-P-1}  There exist $K_4, A_4 > 0$ such that for all $K_0 \ge K_4, A \ge A_4$ and $\gamma \le A^{-4}$ and $ l^* >0$, we can find $ T_4 (K_0, A,l^*) >0$ small enough such that  for all $ \alpha_0, \delta_0, \epsilon_0, \eta_0, C_0 $
 and $T \le T_4,$ and $l \in [0,l^*]$  and under the extra assumptions:
\begin{itemize}
\item  Initial data $u_{d_0,d_1}$ as introduced in \eqref{c4defini-initial-data} with $(d_0,d_1)$ as defined  in Proposition  \ref{c4proposiiton-initial-data} are considered;
\item  $u(t) $ belongs to $S(T,K_0,\epsilon_0, \alpha_0,A,\delta_0,C_0, \eta_0, t)$  for all $ t \in [T-e^{-\sigma}, T-e^{-(\sigma +l)}]$, for some $ \sigma \ge s_0,$ and  $l \in [0,l^*],$
\end{itemize}
then,  the following estimates hold:
\begin{itemize}
\item[$(i)$]    For all $ s \in [\sigma, \sigma +l]$, we  have
 \begin{eqnarray}
 \left|q_0' (s) - q_0(s)  \right|  \le  \frac{C A}{s^\frac{3}{2}},\label{ODE-q-0-'}\\
 \left| q_1'(s) - \frac{1}{2}q_1 (s) \right|  \le \frac{C}{s^2},\label{ODE-q-1-'}\\
 \left|  q_2'(s)  + \frac{2 +\beta}{s} q_2(s) \right|  \le \frac{C A^3}{s^\frac{5}{2}}\label{ODE-q-2-'}.
 \end{eqnarray}
 \item[$(ii)$] For all $s \in [\sigma, \sigma +l]$, we have two cases: 
 
 \noindent
 - If $\sigma >s_0$, then 
 \begin{equation}\label{q-minus-}
 \frac{|q_-(y,s)|}{1+|y|^3} \le   C \left(   A^6 e^{-\frac{s-\sigma}{2}}  +A^7 e^{-(s-\sigma)^2}  +     A^4 e^{(s-\sigma)}(1+(s-\sigma)^2) + (s-\sigma)   \right)   \frac{1}{s^2}.
 \end{equation}
 - If $\sigma = s_0$, then
 \begin{equation}\label{q-minus-s-0}
 \frac{|q_-(y,s)|}{1 +|y|^3}  \le \frac{C(1 + s-\sigma) }{s^2}.
 \end{equation}
\item[$(iii)$]  For all $ s \in [\sigma, \sigma +l]$,   we also have:

\noindent
-    If $ \sigma > s_0$
\begin{eqnarray}
\frac{|\nabla q_\perp (y,s)|}{1+|y|^3} \le \frac{C(C_0)}{s^2} ( A^{6} e^{-\frac{s-\sigma}{2} }   + e^{-(s-\sigma)^2}  + s-\sigma + \sqrt{s-\sigma}),
 \end{eqnarray}
 - If $s = s_0$ then
 \begin{eqnarray*}
\frac{ |\nabla q_\perp (y,s)|}{1+|y|^3}  \le   \frac{C}{s^2}( 1 + s-\sigma + \sqrt{s-\sigma}).
\end{eqnarray*}  
\item[$(iv)$] Finally for all $ s \in [\sigma, \sigma +l] $, we also have:

\noindent
- If  $\sigma > s_0$ then

\begin{equation}
| q_e |   \le  \frac{C}{\sqrt{s}} ( A^7 e^{ \frac{s-\sigma}{p}}  + A^6 e^{s-\sigma}  +A^4 e^{s-\sigma} (1 + (s-\sigma)^2)   + 1 + (s-\sigma)).
\end{equation}

- If $\sigma =s_0$ then
\begin{equation}
|q_e| \le  \frac{C}{\sqrt{s}} (1 + (s-\sigma)).
\end{equation}

\end{itemize}

\end{lemma}
\begin{proof}
The proof is similar to  \cite[Lemma 3.2]{MZnon97}. However,  we also  need to  fine-tune some  important estimates corresponding to the current situation.    For that reason, we will ignore simple estimates and detailed arguments, and instead we will focus on the derivation of important ones.  

$(i)$  In that step, we show estimates  \eqref{ODE-q-0-'} - \eqref{ODE-q-2-'}:    Let us consider $P_j, j=1,2,3,$ the projection  on   the eigenspace  $\mathcal{E}_j, j=1,2,3;$ for more details regarding the definitions of $P_j$ and  $E_j$, see \eqref{c4defin-P-i} and \eqref{eigen-space-E-j}. Now, we apply  $P_j$ to  \eqref{c4equa-Q} and make use of  Lemmas [\ref{lemma-V}-\ref{lemma-G}] to obtain
\begin{eqnarray*}
q_0'  &=& q_0    +  \left(    a- \frac{2b N \kappa}{(p-1)^2} - \frac{\kappa \beta}{(p-1)}   \right)\frac{1}{s}    + \kappa \tilde{\lambda}(s) +    O\left(   \frac{1}{s^2} \right), \\
q_1'  &=& q_1   + O\left( \frac{1}{s^2} \right),\\
q_2'& = &  \frac{q_2}{s} \left(   -\frac{8 b p}{(p-1)^2}  +\frac{ap}{\kappa} - \frac{2 N b p}{(p-1)^2}  - \frac{\beta}{(p-1)}   \right)     - \frac{q_0b p}{(p-1)^2 s}     \\[0.2cm] 
& +&  \frac{1}{s^2} \left( \frac{bp}{(p-1)^2} \left(    \frac{2bN \kappa}{(p-1)^2} - a\right) + \frac{\kappa b}{(p-1)^2} \left( \frac{4pb}{(p-1)^2} - 1 + \frac{\beta}{(p-1)}\right)    \right) \\
&+ &  \tilde{\lambda} q_2 -  \tilde \lambda(s)\frac{ \kappa b }{(p-1)^2s} + O\left( \frac{1}{s^3} \right),
\end{eqnarray*}
where   
\begin{eqnarray*}
b &=& \frac{(p-1)^2}{4p} (1 + \beta),\\
a &=&  \frac{N \kappa}{2ps} (1+\beta) + \frac{\beta \kappa}{(p-1)s},
\end{eqnarray*}
and
$$ \tilde\lambda (s) = \frac{1}{p-1} \left( \frac{\bar \theta'_s(s)}{\bar \theta(s)}  + \frac{\beta}{s}\right).$$
Then, after some  cancellations, we obtain  
\begin{eqnarray*}
q_0'  &=&  q_0 + \kappa  \tilde \lambda(s)+    O\left(   \frac{1}{s^2} \right),\\
q_1 & = & \frac{1}{2} q_1 + O\left( \frac{1}{s^2}\right),\\
q_2'  &=&  \frac{2+\beta}{s} q_2    + \tilde{\lambda}(s) q_2 - \tilde{\lambda}(s) \frac{\kappa b}{(p-1) s} + O\left(\frac{1}{s^3}\right).
\end{eqnarray*}
Also by  Proposition \ref{propo-bar-mu-bounded} we get
$$   |\tilde{\lambda}(s)| \le \frac{CA^3}{s^\frac{3}{2}},  $$
provided that $ \gamma \le A^{-4}$, which concludes item $(i)$.

\medskip
- We now proceed with the proof   of items $(ii)$ and $(iv).$  To this end we  write      \eqref{c4equa-Q} under integral form 
\begin{eqnarray}
q(s) = \mathcal{K}(s,\sigma)(q(\sigma))    + \int_\sigma^s  \mathcal{K}(s,\tau) \left( B(q) + R + G \right)(\tau)  d\tau,\label{integral-equation}
\end{eqnarray}
where   $\mathcal{K}(s,\sigma)$ is the fundamental  solution  associated  to the linear  operator  $\mathcal{L} + \mathcal{V}$, see more details in  \cite{BKnon94}. In particular,  we have the following result (see  \cite{BKnon94} for the $1$-dimensional case, and \cite{NZens16}  for the multidimensional one):  For all  $l^* > 0$,  there exists  $s^*=s^*(l^*)$ such that  for all $ \sigma \ge s^*$ and $v \in L^2_\rho(\R^N),$ then,  for all $s \in [\sigma, \sigma +l^*]$, the function $\psi(s)  = \mathcal{K}(s,\sigma) v$ satisfies
\begin{eqnarray}
\left\|\frac{\psi_-(y,s)}{1+|y|^3} \right\|_{L^\infty }   & \le &  C	\frac{e^{s-\sigma} ((s-\sigma)^2+1) }{s} (|v_0| + |v_1| +  \sqrt{s} |v_2| ) \\
&+& Ce^{-\frac{s -\sigma}{2}} \left\|\frac{v_-(.)}{1 + |y|^3} \right\|_{L^\infty}  + C \frac{e^{-(s-\sigma)^2}}{s^\frac{3}{2}} \|v_e\|_{L^\infty} \nonumber,
\end{eqnarray}
and 
\begin{eqnarray}
\| \psi(s)\|_{L^\infty}  \le Ce^{s-\sigma} \left( \sum_{l=0}^2 s^{\frac{l}{2}} |v_l|   + s^{\frac{3}{2}} \left\|\frac{v_-}{1 +|y|^3} \right\|_{L^\infty} \right)  + Ce^{-\frac{s-\sigma}{p}} \|v_e\|_{L^\infty}.
\end{eqnarray}
Now, we apply this  result to derive
\begin{eqnarray*}
\left|  \int_\sigma^s  \mathcal{K}(s,\tau) \left( B(q) + R + G \right)(\tau)  d\tau    \right|  & \le &  C  \frac{(s-\sigma) }{s^2} (1 + |y|^3),\\
\left\|  \int_\sigma^s  \mathcal{K}(s,\tau) \left( B(q) + R + G \right)(\tau)  d\tau    \right\|_{L^\infty}  & \le & C  \frac{(s-\sigma)}{\sqrt{s}}.
\end{eqnarray*}
It remains  to  obtain the result for $ \psi(s) = \mathcal{K}(s,\sigma) q(\sigma)$. Indeed, using the fact that $q(\sigma) \in V_A(\sigma)$, we obtain
\begin{eqnarray*}
\frac{|\psi_-(y,s)|}{1+|y|^3}  & \le &    C \frac{ e^{s-\sigma} (1 + (s-\sigma)^2)}{s}  \frac{A^4}{\sigma^\frac{3}{2}}  + Ce^{-\frac{s-\sigma}{2}} \frac{A^6}{\sigma^2}   + \frac{C e^{-(s-\sigma)^2}}{\sigma^2} \\
& \le & C \left( e^{-\frac{s-\sigma}{2}} A^6 + e^{-(s-\sigma)^2} A^7+  e^{s-\sigma} (1 + (s-\sigma)^2) A^4  \right)\frac{1}{s^2 },
\end{eqnarray*} 
provided that $ \frac{1}{\sigma} \le \frac{2}{s}$.  
Similarly, we obtain 
\begin{eqnarray*}
\|\psi_e\|_{L^\infty} \le  \frac{C}{\sqrt{s}}(  A^7 e^{-\frac{s-\sigma}{p}}  + e^{s-\sigma} A^6).
\end{eqnarray*}
Therefore we  conclude items $(ii)$ and $(iv)$ for the case $\sigma >s_0$. Furthermore, using  Proposition \ref{c4proposiiton-initial-data}, we also conclude the validity of $(ii)$ and $(iv)$ for the case $\sigma= s_0$.\\
- The proof to item $(iii)$:  We similarly  proceed   as in item $(ii)$   by using  \cite[Lemma B.2 $(iii)$]{MZnon97} and the second estimate in  Definition \ref{defini-shrinking-set-S-t}$(ii)$.
\end{proof}

Next we deal with the following result:

\begin{proposition}[A priori estimates in $P_1 (t)$]\label{c4propo-priori-P-1-later}
There exist $K_5, A_5 \geq 1$ such that for all $K_0 \geq K_5, A \geq A_5, \epsilon_0 > 0, \alpha_0 > 0, \delta_0 \leq \frac{1}{2} \hat{\mathcal{U}}(0), C_0 > 0 , \eta_0 >0 $, and $ \gamma \le A^{-4}$,  there exists $T_5 (K_0, \epsilon_0,\alpha_0,A,\delta_0, C_0,\eta_0)$ such that  for all $T \leq T_5$, the following holds: If $U$ is a non negative  solution  of equation \eqref{equa-U-theta-'-theta} satisfying  $U (t) \in S(T,K_0,\epsilon_0, \alpha_0,A,\delta_0, C_0,\eta_0, t)$ for all $t \in [0,t_5]$ for some $t_5 \in [0,T)$, and  with initial data $ U_{d_0,d_1}$ given in \eqref{c4defini-initial-data} for some $d_0,d_1 \in \mathcal{D}_A$ (cf. Proposition \ref{c4proposiiton-initial-data}), then for all $s \in [ -\ln T,- \ln (T-t_5) ]$, we have the following:
\begin{eqnarray*}
 \forall i,j \in \{1, \cdots, n\}, \quad |q_{2,i,j}(s)| & \leq & \frac{A^4}{2 s^\frac{3}{2}}, \label{c4conq_1-2} \\
\left\| \frac{q_{-}(.,s)}{1 + |y|^3}\right\|_{L^\infty(\mathbb{R}^N)} &\leq &  \frac{A^6}{2 s^{2}}, \left\|  \frac{(\nabla q(.,s))_{\perp}}{ 1 + |y|^3}  \right\|_{L^\infty(\mathbb{R}^N)} \leq  \frac{A^6}{2 s^2}, \\
 \|q_{e}(s)\|_{L^\infty(\mathbb{R}^N)} & \leq &  \frac{A^7}{2 \sqrt s}. \label{c4conq-q-1--and-e}
\end{eqnarray*}
\end{proposition}
\begin{proof}
The proof  is quite similar to \cite[Proposition 3.4]{MZdm97} and  it arises by using Lemma \ref{lemma-priori-P-1}.
\end{proof}
 \subsection{A priori estimates in $P_2$} 
 
In  this subsection,  we show that  the solution satisfies the conditions in $S(t)$, with strict inequalities.
\begin{lemma}[A transmission in $P_2$]\label{lemma-interme-P-2} We can find  $K_6, A_6, \ge 1$ such that      for all $ K_0 \ge K_6, A \ge A_6, \delta_6 >0$, there exist $\alpha_6(K_0,A, \delta_6), C_6(K_0,A, \delta_6) >0$ such that  for all $ \alpha_0 \le \alpha_6, C_0 \ge 2C_6$, there exists $\epsilon_6(\alpha_0,A,\delta_6,C_0)>0$  such that  for all $ \epsilon_0 \le \epsilon_6$, and $\delta_0 \le \frac{1}{3} \left( p-1 + \frac{(p-1)^2}{4p} \frac{K_0^2}{16}\right)^{-\frac{1}{p-1}}$ , there exist  $T_6(\epsilon_0,A,\delta_6,\delta_0,C_0) $ and $ \eta_6(\epsilon_0, A, \delta_6,  \delta_0,C_0)$ such that for all $ T \le T_6, $, we have the following property: assume that $u \in S(T,K_0,\epsilon_0, \alpha_0, A,\delta_0, C_0,\eta_0,t  )$,  $\forall t \in [0,t^*] ,$ for some  $t^* \in [0,T),$ and initial data $u(0)$ defined as in \eqref{c4defini-initial-data} for some $|d_0|, |d_1| \le 2$,  then    the following hold for all $|x| \in \left[ \frac{K_0}{4} \sqrt{(T-t^*)|\ln(T-t^*)|}, \epsilon_0 \right].$
 \begin{itemize}
 \item[$(i)$] For all $\xi \in \frac{7}{4} \alpha_0 \sqrt{|\ln \rho(x)|}$ and $ \tau \in \left[  \max\left(0, \frac{-t(x)}{\rho(x)}\right), \frac{t^* -t(x)}{\rho(x)} \right]$ the following estimates are valid
 \begin{eqnarray*}
 \left|\nabla_\xi \mathcal{U}(x, \xi, \tau) \right|  \le \frac{C(K_0,C_0)A^2}{\sqrt{|\ln \rho(x)|}},\\
 \frac{1}{4} \left( p-1 + \frac{(p-1)^2}{4p} \frac{K_0^2}{16}\right)^{-\frac{1}{p-1}}  & \le & \mathcal{U}(x,\xi,\tau) \le C(K_0).
 \end{eqnarray*}
 \item[$(ii)$] For all $ |\xi| \le 2\alpha_0 \sqrt{|\ln\rho(x)|}$ and $ \tau_0 = \max\left(0,\frac{t(x)}{\rho(x)} \right)$, we have
 $$  \left|   \mathcal{U}(x,\xi, \tau)  - \mathcal{\hat U}(x,\tau_0)  \right| \le \delta_6   \text{ and } \left| \nabla_\xi \mathcal{U}(x,\xi,\tau) \right| \le  \frac{C_6}{\sqrt{|\ln \rho(x)|}}. $$
 \end{itemize}
\end{lemma}
\begin{proof}
  The proof can be found in  \cite[Lemma 2.6]{MZnon97}.
\end{proof}
Now, we use the parabolic regularity to derive the following:
\begin{proposition}[A priori estimate in $P_2(t)$]\label{priori-estimate-P-2} There exist $K_7 \ge 1, A_7 \ge 1$ such that for all $K_0 \ge K_7,A \ge A_7 $, there exist $ \delta_7 \le \frac{1}{3} \left(p-1 + b \frac{K_0^2}{16} \right)^{-\frac{1}{p-1}}$  and $C_7(K_0,A)$ such  that  for all $ \delta_0 \le \delta_7, C_0 \ge C_7$ there exists  $ \epsilon_7(K_0, \delta_0,C_0), \alpha_7(K_0,A,C_0)$  such that $ \epsilon_0 \le \epsilon_7$ there exists $ T_7(K_0,A,\delta_0, C_0, \epsilon_0), \eta_7 >0$ such that for all $T \le T_7, \eta_0 \le \eta_7$, we have the following property: assume that  $ u  \in S(T,K_0,\epsilon_0,\alpha_0, A, \delta_0, C_0, t)  $   for  all $ t \in  [0,t^*]$ for some $t^* \in [0,T)$, and initial data $u(0)$ defined as in \eqref{c4defini-initial-data} for some $|d_0|, |d_1| \le 2$, 0then for all $ |x| \in  \left[  \frac{K_0}{4} \sqrt{(T-t^*)|\ln(T-t^*)|}, \epsilon_0 \right],  $ $ |\xi| \le \alpha_0 \sqrt{|\ln \rho(x)|}$ and $\tau^*= \frac{t^*- t(x)}{\rho(x)}  $, we have
\begin{equation*}
\left|   \mathcal{U}(x, \xi, \tau^*) - \hat{\mathcal{U}}(x,\tau^*) \right| \le \frac{\delta_0}{2} \text{ and }  \left| \nabla_\xi   \mathcal{U}(x,\xi,\tau^*) \right|    \le \frac{C_0}{2 \sqrt{ |\ln(\rho(x))|}}.
\end{equation*}
\end{proposition} 
\begin{proof}
Note that  $  \mathcal{U}$ satisfies 
\begin{equation*}
\partial_\tau \mathcal{U} = \Delta \mathcal{U}   + \tilde{\theta}(\tau) \theta^{-1}(t(x))  \mathcal{U}^p - \rho(x) \mathcal{U}.    
\end{equation*} 
By $\theta$'s  monotonicity, we conclude that for all  $\tau' \in [\tau_0(x),\tau^*(t^*,x) ]$
$$  \left|\tilde{\theta}(\tau') \theta^{-1}(t(x)) \right| \le 1.$$
We also have
$$  \partial_\tau \mathcal{\hat U}  =  \tilde{\theta}(\tau) \theta^{-1}(t(x)) \mathcal{\hat U}^p,  $$
and if we set
$$ \mathcal{\bar U}(\tau): = \mathcal{U}(\tau)  - \mathcal{\hat U}(\tau) ,    $$
then $ \mathcal{\bar U}$ satisfies
$$   \partial_\tau \mathcal{\bar U}  = \Delta \mathcal{\bar U}  +  \tilde{\theta}(\tau) \theta^{-1}(t(x)) \left[ \mathcal{ U}^p -\mathcal{\hat U}^p  \right]  -  \rho(x) \mathcal{U}.  $$
We remark that
\begin{eqnarray*}
 \left|    \tilde{\theta}(\tau) \theta^{-1}(t(x))(  \mathcal{ U}^p - \mathcal{\hat U}^p ) \right| \le C \mathcal{\bar U}.
\end{eqnarray*}
We now consider $\chi_{\frac{7}{4}}$ which is smooth, satisfying
 $\chi_{\frac{7}{4}} (r) = 1, \forall  |r | \le   [0,1] $ and $\chi_{\frac{7}{4}}( r ) =0, \forall |r| \ge \frac{7}{4}$.  Next we define
 $$   \tilde \chi( \xi) := \chi_{\frac{7}{4}}\left(\frac{|\xi|}{ \alpha_0\sqrt{|\ln(\rho(x))|}} \right) ,  $$
and
$$  \mathcal{\tilde U} := \tilde \chi  \mathcal{\bar U}.  $$
Hence, we derive
\begin{eqnarray}
\partial_\tau  \mathcal{\tilde U} =    \Delta \mathcal{\tilde U}   + \tilde \chi \tilde{\theta}(\tau) \theta^{-1}(t(x))(  \mathcal{ U}^p -\mathcal{\hat U}^p) + G(x,\xi,\mathcal{U},\tau),\label{equa-tilde-mathcal-U}
\end{eqnarray}
where
\begin{eqnarray*}
\left|   \tilde \chi \tilde{\theta}(\tau) \theta^{-1}(t(x))(  \mathcal{ U}^p -\mathcal{\hat U}^p)         \right| \le C \mathcal{\tilde{U}}, 
\end{eqnarray*}
and 
\begin{eqnarray*}
\left|G(x,\xi,\mathcal{U},\tau) \right| \le \frac{C}{\sqrt{|\ln(\rho(x))|}}.
\end{eqnarray*}
By virtue of   \eqref{equa-tilde-mathcal-U} we obtain, for all $\tau \in [\tau_0, \tau^*],$ the following integral equation: 
$$ \mathcal{\tilde U} (\tau)=  e^{(\tau -\tau_0) \Delta } \mathcal{\tilde U}(\tau_0)  + \int_{\tau_0}^\tau   e^{(\tau_0 - \tau') \Delta}  (  \tilde \chi \tilde{\theta}(\tau') \theta^{-1}(t(x))(  \mathcal{ U}^p -\mathcal{\hat U}^p) (\tau')+ G(x,\xi,\mathcal{U},\tau')   )  d\tau'. $$     
Taking the $L^\infty-$ norm  to the above equation in applying Lemma \ref{lemma-interme-P-2},  we have
\begin{eqnarray*}
\|   \mathcal{\tilde U}(\tau)     \|_{L^\infty}  \le  C( \delta_6 + \frac{1}{\sqrt{|\ln\rho (x)|}}  )  + C \int_{\tau_0}^\tau \|   \mathcal{\tilde U}(\tau')     \|_{L^\infty},  
\end{eqnarray*}
and via Gronwall's lemma we finally obtain
$$ \|   \mathcal{\tilde U}(\tau)     \|_{L^\infty}     \le  3C\left( \delta_6 + \frac{1}{\sqrt{|\ln\rho (x)|}}  \right) \le  \frac{\delta_0}{2}, $$
provided that $ \delta_6 \le \delta_7(\delta_0)$ and $  \epsilon_0 \le \epsilon_7$. 
The proof of the estimate for $\nabla_\xi \mathcal{U}$ follows the same technique
 and so it is omitted. 
 \end{proof}
 
 \subsection{A priori estimates on $P_3$}
 Next we prove that the conditions in  Definition \ref{defini-shrinking-set-S-t} $(iii)$ are strictly   satisfied.
Indeed the following holds: 
 \begin{proposition}[A priori estimate in $P_3(t)$]\label{priori-estimate-P-3} Let us consider $K_0, \epsilon_0, \alpha_0, A, C_0, \eta_0>0$ and $\delta_0, \in  \left( 0, \frac{1}{3} \left(p -1  + b\frac{K_0^2}{16} \right)^{-\frac{1}{p-1}}   \right]$. Then, there exists $T_8 >0$ small enough such that for all $T\in (0,T_8)$, the following property  holds: assume that $u$ be a non negative solution to \eqref{equa-u-non-local} for  all $t \in [0,t^*]$ and it is satisfied that $u(t)\in S(T,K_0,\epsilon_0, \alpha_0, A, \delta_0, C_0,\eta_0, t)$ for all $ t \in [0,t^*]$ corresponding to initial data $u_{d_0,d_1}$   for some $|d_0|, |d_1| \le 2,$ then  for all $x \in \Omega \cap \{ |x| \ge \frac{\epsilon_0}{4}  \}$, we have  
 $$  \left| u(x,t^*)  - u(0)   \right| \le \frac{\eta_0}{2},  $$
 and
 $$  \left| \nabla u(x,t^*)  - \nabla e^{ t^* \Delta}  u(0)   \right|   \le \frac{\eta_0}{2}, $$
 where $e^{t\Delta}$ is  the semi-group associated to $\Delta$ with Neumann boundary  conditions.
 \end{proposition}
 \begin{proof}
 We write \eqref{equa-u-non-local} as follows:
 \begin{equation*}
 u(t) =  e^{t \Delta } u(0) + \int_0^t e^{(t-t')\Delta} \left[ \theta(t') u^p  - u(t') \right] dt',
 \end{equation*}
where  $e^{t\Delta}$ is the semi-group associated to Neumann boundary  conditions, see more in \cite{YZPA13}. Thus, the proof directly comes from  parabolic regularity estimates. 
 \end{proof}

\section{Computation on initial data}\label{proof-initial-data-propo}
In the current section we provide the complete proof of Proposition \ref{c4proposiiton-initial-data}. 

Indeed, we first prove that such a initial data considered by Proposition \ref{c4proposiiton-initial-data} will ensure the following asymptotic behavior
\begin{equation}\label{proof-theta-0-}
\theta(0) = \theta_\infty |\ln(T)|^{-\beta} \left( 1+ O\left( \frac{1}{\sqrt{|\ln T|}} \right)\right)\quad\mbox{as}\quad t\to T,
\end{equation}
where  $\theta_\infty$ and $\beta$ defined as in \eqref{defi-theta-infty} and  \eqref{defi-beta}, respectively.

Converting to $\bar{\theta}(s_0), s=-\ln(T-t) , s_0 = -\ln T$, defined as in  
\eqref{defi-bar-theta-s} then equivalently we must ensure that
\begin{eqnarray*}
\bar{\theta}(s_0) = \theta_{\infty} s_0^{-\beta} \left(  1 + O\left(\frac{1}{\sqrt{s_0}} \right) \right).   
\end{eqnarray*}
In particular, we use the relation in \eqref{defi-theta-}, we claim that the following derives \eqref{proof-theta-0-}
\begin{equation}\label{goal--theta-0}
    \|u_{d_0,d_1}\|^r =\theta_\infty^{-\frac{N}{2}} \left[\frac{b}{2} \right]^{-\frac{N}{2}}\frac{1}{2(1 +\frac{N}{2}(1+\beta))} \left(  |\ln T|\right)^{1 +\frac{N}{2} (1+\beta)} \left(1 + O\left( \frac{1}{\sqrt{|\ln T|}}  \right)\right).  
\end{equation}
Indeed, let us suppose that \eqref{goal--theta-0} holds. Then, from using $\theta$'s definition we also derive
\begin{eqnarray*}
& & \theta (0) = |\Omega|^\gamma \left( \|u_{d_0,d_1}\|^r_{L^r} \right)^{-\gamma}\\
& =& |\Omega|^{\gamma} \theta_\infty^{ \frac{\gamma N}{2}} \left[ \frac{b}{2} \right]^{ \frac{\gamma N}{2}}2^{\gamma}\left(1+\frac{N}{2}(1+\beta)\right)^\gamma|\ln T|^{-\gamma (1+ \frac{N}{2}(1 +\beta))}\left(1 + O\left(\frac{1}{\sqrt{|\ln T|}}  \right)\right).
\end{eqnarray*}
In  fact
$$ -\gamma \left(1+ \frac{N}{2}(1 +\beta)\right)=-\beta, $$
whilst thanks to  Lemma \eqref{bubles-integral} we have
$$ \int_{0}^{\infty } \varphi_0^{p-1+r} (\xi) \xi^{N-1}  d\xi = \frac{1}{(p-1)N} b^{-\frac{N}{2}},$$
and so
\begin{eqnarray}
\theta_\infty = \left( \frac{|\Omega| \left(1 +\frac{N}{2} \right)}{1- \frac{\gamma N}{2}} 2b^\frac{N}{2}\right)^\frac{\gamma}{1- \frac{\gamma N}{2}}.
\end{eqnarray}
The latter infers 
$$|\Omega|^{\gamma} \theta_\infty^{\gamma \frac{N}{2}} \left[ \frac{b}{2} \right]^{\gamma \frac{N}{2}}2^{\gamma}\left(1+\frac{N}{2}(1+\beta)\right)^\gamma  =\theta_\infty, $$
which finally implies  \eqref{proof-theta-0-}.
\medskip

Now, we start the proof of \eqref{goal--theta-0}.   Let us consider $(d_0,d_1) \in [-2,2]^{1+N}$, $\frac{r}{p-1} = \frac{N}{2}$ and $p \ge 3$. We first provide an estimate for  $\|u_{d_0,d_1}\|^r_{L^r(\Omega)}.$ Recalling its definition we get
\begin{eqnarray*}
\|u_{d_0,d_1}\|^r_{L^r(\Omega)} = \int_{\Omega} u_{d_0,d_1}^r dx &=& \int_{|x| \le K_0 \sqrt{T |\ln T|} |\ln T|^\frac{1}{2}} u_{d_0,d_1}(0)^r dx \\
&+&  \int_{|x| \ge  K_0\sqrt{T|\ln T|}|\ln T|^\frac{1}{2}, x\in \Omega} u_{d_0,d_1}(0)^r dx \\
& = & I_1 + I_2. 
\end{eqnarray*}
Hence,\eqref{goal--theta-0}   follows   the following
\begin{eqnarray}
I_1   &  \le & C  |\ln T|^{\frac{N}{2} +\frac{r \beta}{p-1}} |\ln|\ln T|| \label{esti-I-1-intial} ,\\
I_2  &=& \theta_\infty^{-\frac{N}{2}} \left[\frac{b}{2} \right]^{-\frac{N}{2}}\frac{1}{2(1 +\frac{N}{2}(1+\beta))} \left(  |\ln T|\right)^{1 +\frac{N}{2} (1+\beta)} \left(1 + O\left( \frac{|\ln|\ln T||}{|\ln T|}  \right)\right). \label{esti-I-2-intial}
\end{eqnarray}
\begin{itemize}
\item Estimation for $I_1$:  Let us recall $u_{d_0,d_1} $ defined as in   \eqref{c4defini-initial-data}, then, for all $ |x| \le  K_0 \sqrt{T|\ln T|}|\ln T|^\frac{1}{2}$, we have
$$ u_{d_0,d_1} = T^{-\frac{1}{p-1} } \left(\theta_\infty |\ln T|^{-\beta} \right)^{-\frac{1}{p-1}}   \left[  \varphi \left( \frac{x}{\sqrt{T}}, - \ln s_0  \right)   +  (d_0 \frac{A^3}{s_0^\frac{3}{2}}+  \frac{A}{s^2_0} d_1  \cdot  \frac{x}{\sqrt{T}}   )  \chi_0 \left(  \frac{|z_0|}{\frac{K_0}{32}}\right) \right] .  $$
We now decompose $I_1$  as follows
\begin{eqnarray*}
I_1 = \int_{|x| \le \frac{K_0}{16} \sqrt{T|\ln T|} } u_{d_0,d_1}^r(0) dx + \int_{ \frac{K_0}{16} \sqrt{T|\ln T|} \le |x| \le K_0\sqrt{T|\ln T|}|\ln T|^\frac{1}{2} } u_{d_0,d_1}^r(0) dx.
\end{eqnarray*}
We see that for all $ |x| \le \frac{K_0}{16} \sqrt{T|\ln T|} $, 
\begin{eqnarray*}
\left| u^r_{d_0,d_1}(0) \right| \le C(K_0) T^{-\frac{r}{p-1}} |\ln T|^{\frac{r \beta}{p-1}}.
\end{eqnarray*}
Then we obtain
\begin{eqnarray*}
\left|  \int_{|x| \le \frac{K_0}{16} \sqrt{T|\ln T|} }  u_{d_0,d_1}^r(0) dx     \right|   \le C(K_0) |\ln T|^{\frac{N}{2} + \frac{r\beta}{p-1}} .
\end{eqnarray*}
Now we consider $ |x| \in \left[ \frac{K_0}{16} \sqrt{T|\ln T|}, K_0 \sqrt{T|\ln T|} |\ln T|^\frac{1}{2}   \right] $ and thus 
$$ u_{d_0,d_1}( 0) = T^{-\frac{1}{p-1}}(\theta_\infty )^{-\frac{1}{p-1}} |\ln T |^{\frac{\beta}{p-1}}   \left[ \varphi_0\left( \frac{|x|}{\sqrt{T|\ln T}}\right) + \frac{a}{|\ln T|} \right].$$
At this point we recall the following inequality
\begin{equation}\label{inequality-ab-b}
\left| (a + b)^\alpha  \right| \le C(\alpha)\left[  a^{\alpha} + b^\alpha \right],
\end{equation}
for  $a,b,\alpha > 0$.  So, we have
\begin{eqnarray*}
\left[ \varphi_0\left( \frac{|x|}{\sqrt{T|\ln T}}\right) + \frac{a}{|\ln T|} \right]^r \le C(r)\left[  \varphi_0^{r} \left( \frac{|x|}{\sqrt{T|\ln T}}\right)  +  \frac{1}{|\ln T|^r}       \right], 
\end{eqnarray*}
which yields
\begin{eqnarray*}
& & \left| \int_{ \frac{K_0}{16} \sqrt{T|\ln T|} \le |x| \le K_0\sqrt{T|\ln T|}|\ln T|^\frac{1}{2} } u_{d_0,d_1}^r(0) dx \right|\\
& \le & |\ln T|^{\frac{N}{2} + \frac{r\beta}{p-1} } \left(\int_{\frac{K_0}{16}}^{K_0 \sqrt{|\ln T|}} \left[ \varphi_0^{r} (\xi)\xi^{N-1}   + \frac{1}{|\ln T|^r}\right] d\xi \right)\\
& \le & C |\ln T|^{\frac{N}{2} +\frac{r\beta}{p-1} } |\ln|\ln T||.  
\end{eqnarray*}
and finally   \eqref{esti-I-1-intial} arises.

\item Behavior of $ I_2$: For all $|x| \in \left[ K_0 \sqrt{T|\ln T} |\ln T|^\frac{1}{2}, 2K_0 \sqrt{T|\ln T} |\ln T|^\frac{1}{2} \right] $,  $u_{d_0,d_1}(0)$ has the following form
\begin{eqnarray*}
u_{d_0,d_1}(0) =  \chi_1(x,0) \left( T^{-\frac{1}{p-1}} \theta_\infty^{-\frac{1}{p-1}} |\ln T|^\frac{\beta}{p-1} \varphi_0 \left(\frac{|x|}{\sqrt{T|\ln T|}}\right) - H^*(x) \right)  + H^*(x).
\end{eqnarray*}
It is  easy to see that 
\begin{eqnarray*}
& & \left| T^{-\frac{1}{p-1}} \theta_\infty^{-\frac{1}{p-1}} |\ln T|^{\frac{\beta}{p-1}} \varphi_0 \left(\frac{|x|}{\sqrt{T|\ln T|}} \right)       \right| \le  C H^*(x),
\end{eqnarray*}
which yields
$$|u_{d_0,d_1}| \le 2C H^*,$$
and thus we obtain
\begin{eqnarray*}
& & \int_{K_0\sqrt{T|\ln T|}|\ln T|^\frac{1}{2} \le |x| \le 2K_0\sqrt{T|\ln T|}|\ln T|^\frac{1}{2}}  u_{d_0,d_1}^r(0) dx \\
& \le &C \int_{K_0\sqrt{T|\ln T|}|\ln T|^\frac{1}{2} \le |x| \le 2K_0\sqrt{T|\ln T|}|\ln T|^\frac{1}{2}}  ( H^{*})^r dx\\
&\le & C \int_{K_0\sqrt{T|\ln T|}|\ln T|^\frac{1}{2} \le |x| \le 2K_0\sqrt{T|\ln T|}|\ln T|^\frac{1}{2}}  [x^2]^{-\frac{r}{p-1}} |\ln|x||^{\frac{N}{2}(1+\beta)} dx\\
& \le & C |\ln |\ln T||^{1 + \frac{N}{2}(1 +\beta)} .
\end{eqnarray*}
We also have
$$\int_{|x| \ge \epsilon_0, x \in \Omega }  u_{d_0,d_1}^r  dx \le C(\epsilon_0), $$
and thus it remains to estimate the following integral
\begin{eqnarray*}
& &\int_{2K_0\sqrt{T|\ln T|} |\ln T|^\frac{1}{2} \le |x| \le \epsilon_0 } u_{d_0,d_1}^r dx = \int_{2K_0\sqrt{T|\ln T|} |\ln T|^\frac{1}{2} \le |x| \le \epsilon_0 } (H^*)^r dx.
\end{eqnarray*}
\begin{eqnarray*}
&=& \int_{2K_0\sqrt{T|\ln T|}|\ln T|^\frac{1}{2} \le |x| \le \epsilon_0} \left[ \theta_\infty \frac{b}{2}\right]^{-\frac{N}{2}} |x|^{-\frac{N}{2}} |2\ln|x||^{\frac{N}{2}(1+\beta)} dx \\
&+& O(|\ln T|^{\frac{N}{2}(1+\beta)}|\ln|\ln T||)\\
& = &\theta_\infty^{-\frac{N}{2}} \left[\frac{b}{2} \right]^{-\frac{N}{2}} \left(  |\ln T|\right)^{1 +\frac{N}{2} (1+\beta)} \frac{1}{2(1+\frac{N}{2}(1+\beta))} \left(1 + O\left( \frac{|\ln|\ln T||}{|\ln T|}  \right)\right).
\end{eqnarray*}
Consequently we derive
\begin{eqnarray*}
& &\int_{2K_0\sqrt{t|\ln T|}|\ln T|^\frac{1}{2} \le |x| \le \epsilon_0 } u_{d_0,d_1}^r dx \\
& = &\theta_\infty^{-\frac{N}{2}} \left[\frac{b}{2} \right]^{-\frac{N}{2}} \left(  |\ln T|\right)^{1 +\frac{N}{2} (1+\beta)} \frac{1}{2(1+\frac{N}{2}(1+\beta))} \left(1 + O\left( \frac{|\ln|\ln T||}{|\ln T|}  \right)\right).
\end{eqnarray*}
which concludes \eqref{esti-I-2-intial}.
\end{itemize}
\medskip
In particular, we finish the proof of \eqref{goal--theta-0}.

Now, we start the proof  of  item $({\it I})$ of Proposition \ref{proof-initial-data-propo}:

\begin{itemize}
\item  Estimates on  the problem in  similarity variables: In this part,  we estimate the solution expressed in the  variable  $(y,s)$ as in \eqref{similarity-variable}. Following the chain of definition
$u_{d_0,d_1}(0) \to U_{d_0,d_1}(0) \to W_{d_0,d_1} \to q(s_0)$ as   in  \eqref{defi-U-0}, \eqref{defi-W-s-0} and \eqref{defi-q-y-0-s-0}, 
we derive the following:
\begin{eqnarray*}
U_{d_0,d_1}(0) & = & \theta^\frac{1}{p-1}(0) \chi_1(0) u_{d,_0,d_1}(0)\\
& = &  (\theta_\infty |\ln T|^{-\beta})^\frac{1}{p-1}\left(1 + O\left(\frac{1}{\sqrt{|\ln T|}} \right) \right) \chi_1(0)  u_{d_0,d_1}(0)\\
& = & T^{-\frac{1}{p-1}} \left[ \varphi\left(\frac{x}{\sqrt{T}}, -\ln T \right) + \left( d_0 \frac{A^3}{s_0^\frac{3}{2}} + \frac{A}{s_0^2} d_1 \cdot y_0 \right) \chi_0 \left( \frac{32 z_0}{K_0}\right)   \right] \chi_1^2(0) \\
& +& H^*(x) (1 -\chi_1(0))\chi_1(0) \theta^\frac{1}{p-1}(0) \\
&+& O\left( \frac{T^\frac{-1}{p-1}}{\sqrt{|\ln T|}}\right) \left[ \varphi\left(\frac{x}{\sqrt{T}}, -\ln T \right) + \left( d_0 \frac{A^3}{s_0^\frac{3}{2}} + \frac{A}{s_0^2} d_1 \cdot y_0 \right) \chi_0 \left( \frac{32 z_0}{K_0}\right)   \right] \chi_1^2(0) ,  
\end{eqnarray*}
where $y_0 = \frac{x}{\sqrt{T}}$ and $z_0 = \frac{x}{\sqrt{T|\ln T|}}$. Then, we  derive initial data $W(y_0,s_0)$, a function with $y_0$ variable
\begin{eqnarray*}
W_{d_0,d_1}(y_0,s_0) &=& \left( \varphi(y,s_0)   +  \left( d_0 \frac{A^3}{s_0^\frac{3}{2}} + \frac{A}{s_0^2} d_1 \cdot y_0 \right) \chi_0 \left( \frac{32 z_0}{K_0}\right) \right) \chi_1^2(y_0,s_0)\\
& +  & T^\frac{1}{p-1} H^{*} ( y_0 \sqrt{T} ) (1-\chi_1(y_0,s_0)) \chi_1(y_0,s_0) \theta^\frac{1}{p-1} (0)\\
& + & O\left( \frac{1}{\sqrt{|\ln T|}}\right) \left[ \varphi\left(\frac{x}{\sqrt{T}}, -\ln T \right) + \left( d_0 \frac{A^3}{s_0^\frac{3}{2}} + \frac{A}{s_0^2} d_1 \cdot y_0 \right) \chi_0 \left( \frac{32 z_0}{K_0}\right)   \right] \chi_1^2(y_0,s_0)
\\
 & + &  \varphi(y,s_0)   +  \left( d_0 \frac{A^3}{s_0^\frac{3}{2}} + \frac{A}{s_0^2} d_1 \cdot y_0 \right) \chi_0 \left( \frac{32 z_0}{K_0}\right) +     \tilde{W}(y_0,s_0),
\end{eqnarray*}
where $\tilde W$ defined by
\begin{eqnarray*}
\tilde{W}(y_0,s_0) & = &  \left( \varphi(y,s_0)   +  \left( d_0 \frac{A^3}{s_0^\frac{3}{2}} + \frac{A}{s_0^2} d_1 \cdot y_0 \right) \chi_0 \left( \frac{32 z_0}{K_0}\right) \right) (\chi_1^2(y_0,s_0)-1) \\
 & +  & T^\frac{1}{p-1} H^{*} ( y_0 \sqrt{T} ) (1-\chi_1(y_0,s_0)) \chi_1(y_0,s_0) \theta^\frac{1}{p-1} (0)\\
& + & O\left( \frac{1}{\sqrt{|\ln T|}}\right) \left[ \varphi\left(\frac{x}{\sqrt{T}}, -\ln T \right) + \left( d_0 \frac{A^3}{s_0^\frac{3}{2}} + \frac{A}{s_0^2} d_1 \cdot y_0 \right) \chi_0 \left( \frac{32 z_0}{K_0}\right)   \right] \chi_1^2(y_0,s_0).
\end{eqnarray*}
It is easy to have
$$ \|\tilde{W}\|_{W^{1,\infty}_{y_0}} \le \frac{C}{\sqrt{|\ln T|}}.$$  
As a matter of fact,  $q(y_0, s_0)$   will be the following
\begin{eqnarray*}
q(y_0,s_0) =    \left( d_0 \frac{A^3}{s_0^\frac{3}{2}} + \frac{A}{s_0^2} d_1 \cdot y_0 \right) \chi_0 \left( \frac{32 z_0}{K_0}\right)  + \tilde W(y_0,s_0).
\end{eqnarray*}
Up to a small perturbation $\tilde W(y_0,s_0) $, $q(y_0,s_0)$  is the same as in \cite[Lemma 2.4 ]{MZnon97}.  

\item Estimate in $P_2(0):$    Let us consider $ |x| \in \left[\frac{K_0}{4} \sqrt{T|\ln T|},  \epsilon_0 \right]$, and recall the definition of 
$$      \mathcal{U}(x,\xi,\tau_0) =(T-t(x))^{\frac{1}{p-1}}  \theta^{\frac{1}{p-1}}(t(x)) u(x,0),    $$
and
\begin{equation*}
\hat{\mathcal{U}}(x,\tau_0(x)) = \left((p-1)(1-\int_0^{\tau_0(x)} \tilde \theta(\tau') (\theta(t(x)))^{-1} d\tau' \right)^{-\frac{1}{p-1}},
\end{equation*}
where
$$ \tau_0(x) = \frac{-t(x)}{T-t(x)} \in [ 0,1].$$
The reader should bear in mind that 
$$ \theta(t(x)) = \theta(0) \text{ if  } t(x) \le  0 ,  $$ 
with $\theta(0)$ satisfying \eqref{proof-theta-0-}. Now,  we  observe that for  all $ \tau' \in [0,\tau_0(x)]$ and $  |x| \in \left[\frac{K_0}{4} \sqrt{T|\ln T|} ,  \epsilon_0 \right]$ 
$$  \tau' \rho + t(x) \le 0 , $$
which implies
\begin{eqnarray}
\hat{\mathcal{U}}(x,\tau_0(x)) &=& \left((p-1)(1-\int_0^{\tau_0(x)} \theta(0) (\theta(0))^{-1} d\tau' + b \frac{K_0^2}{16}\right)^{-\frac{1}{p-1}} \nonumber\\
&  = &  \left( (p-1)\left(1-\tau_0(x)  \right) + b\frac{K_0^2}{16}   \right)^{-\frac{1}{p-1}}.\;\;\label{defi-mathcal-hat-U-tau-0}
\end{eqnarray}

Now, we will write    $\mathcal{U}(\tau_0)$  with $\tau_0(x) = \frac{-t(x)}{T-t(x)}$  as follows 
\begin{eqnarray*}
\mathcal{U}(x,\xi, \tau_0) &=& (T-t(x))^{\frac{1}{p-1}} (\theta(t(x)))^{\frac{1}{p-1}} u (x+ \xi \sqrt{T-t(x)},0)
\\
& = &\left( \frac{(T-t(x)) \theta(t(x)) }{T  \theta_\infty |\ln T|^{-\beta} } \right)^{\frac{1}{p-1}} \left(       p-1 +  b \frac{  |x+ \xi \sqrt{T-t(x)}|^2}{T|\ln T|} \right)^{-\frac{1}{p-1}} \chi_1(  x+ \xi \sqrt{T-t(x)},0  )\\
&+& (T-t(x))^\frac{1}{p-1} \theta^\frac{1}{p-1}(t(x)) H^*(x+\xi \sqrt{T-t(x)}) (1 - \chi_1(  x+ \xi \sqrt{T-t(x)},0  ))\\
& = & I  \chi_1(  x+ \xi \sqrt{T-t(x)},0  )  + II (1-\chi_1(  x+ \xi \sqrt{T-t(x)},0  )).
\end{eqnarray*}
 Let  us  mention that,  our functions $\mathcal{U}(x,\xi,\tau_0)$ and $\hat{\mathcal{U}}(x,\tau_0(x))$  are similar to  the  ones at     \cite[page 1531]{MZnon97}. So, we can apply the process     to prove the following estimates (see more details in that work):
 
- For all $|x|  \in \left[  \frac{K_0}{4} \sqrt{T|\ln T|}, 2K_0     \sqrt{T|\ln T|} |\ln T|^\frac{1}{2} \right]$ and $| \xi| \le 2 \alpha_0 \sqrt{|\ln \rho(x)|}$ 
\begin{equation}\label{estimate-I-initial-data}
\left|   I    -   \hat{\mathcal{U}} (x, \tau_0 (x) )            \right| \le  \frac{\delta_3}{2}.
\end{equation}

- For all   $|x| \in  \left[      K_0     \sqrt{T|\ln T|} |\ln T|^\frac{1}{2}, \epsilon_0               \right]$ and $ |\xi| \le 2\alpha_0 \sqrt{|\ln \rho(x)|} $

\begin{equation}\label{estimate-II-initial-data}
\left|   II     -     \hat{\mathcal{U}}(x,\tau_0)       \right|  \le \frac{\delta_3}{2}.
\end{equation}
 From   \eqref{estimate-I-initial-data} and \eqref{estimate-II-initial-data}, we conclude 
$$   \left|      \mathcal{U}(x,\xi,\tau_0) -  \hat{\mathcal{U}} (x,\tau_0)   \right|   \le \delta_3.   $$ 
In addition to that,  the technique at  \cite[pages 1533-1535]{MZnon97} can be applied to prove the following:
$$  \left|   \nabla_\xi \mathcal{U}(x,\xi,\tau_0)    \right|\leq    \frac{C_3}{\sqrt{|\ln\rho(x)|}}.$$
Finally,    the proof of the required estimates in $P_2(0)$ immediately follows.

 \end{itemize}
For item $(II)$ obviously arises from item $(I)$, see more details    in \cite{TZpre15}. We finish the proof of the Proposition.

 \appendix
 \section{ Necessary estimates }\label{appendix-A}
 
In this part, we aim to give more details on some rough calculations presented in the preceding sections, and  support  the proof of our main result.

 \begin{lemma}[Potential term $V$]\label{lemma-V} Let us consider $V$ defined as in   \eqref{c4defini-potential-V}. Then, we have the following estimates
 \begin{eqnarray*}
 V (y,s) &=&   - \frac{p b}{(p-1)^2}  ( |y|^2  - 2N)     + \frac{1}{s} \left( \frac{ap}{\kappa} -  \frac{2N b p}{(p-1)^2}  \right)   \\
&+& O\left( \frac{1+|y|^4}{s^2}\right), 
 \end{eqnarray*}
 where 
$$  \left| \tilde V (y,s)\right| \le \frac{C(K_0)(1+|y|^2)}{s^2}, \forall |y| \le 2K_0 \sqrt{s},$$
and 
\begin{eqnarray}
a =  \frac{2b N \kappa}{(p-1)^2}  +  \frac{\kappa \beta}{(p-1)}.\label{defi-a}
\end{eqnarray}
Additionally we  have
$$ |V(y,s)| \le C, \text{ for all } y \in \R^N.$$
 \end{lemma}
 \begin{proof}
 The proof arises directly from  a Taylor expansion and $V$'s definition.
  \end{proof}
 Next, we  give some estimates on  $B$ defined by \eqref{c4defini-B-Q}.
 \begin{lemma}[Quadratic term $B$] \label{lemma-B}
 Let us  consider $B(q)$ defined as in    \eqref{c4defini-B-Q} and $q \in V_A(s)$ where $V_A$ introduced  in Definition   \ref{defini-shrinking-set-S-t}. Then, the following hold
 \begin{eqnarray*}
\left|B(q) \right|  &\le & C (K_0) |q|^2,  \forall |y| \le 2K_0 \sqrt{s},\\
| B(q) |  & \le & C |q|^{\bar p}, \forall y \in \R^N, \text{where }  \bar p = \min(p,2).
 \end{eqnarray*}
 \end{lemma}
 \begin{proof}
 The proof is pretty the same as the proof of an analogous result in \cite{MZdm97}.
 \end{proof}
 Now, we will  study on the  rest term, $R$
 \begin{lemma}[The rest term  $R$] 
 Let us consider $R$ defined as in  \eqref{c4defini-rest-term}. Then, we have the following estimates
\begin{eqnarray}
\|R(.,s)\|_{L^\infty}  & \le &  \frac{C}{s},\label{estima-R-infty}
\end{eqnarray} 
and 
\begin{eqnarray*}
R(y,s) &=&  \left(    a - 2 b \frac{N  \kappa}{(p-1)^2}   \right)  \frac{1}{s} + \frac{a_0}{s^2}  \\
&+&   \frac{|y|^2}{s^2} \left(  \frac{b}{(p-1)^2} \left(    \frac{2bN \kappa}{(p-1)^2} - a\right) + \frac{\kappa b}{(p-1)^2} \left( \frac{4pb}{(p-1)^2} - 1\right)  \right)        \\
& + & O\left(\frac{1 +|y|^4}{s^3}\right)
 .\label{R-expansion-gerenral}
\end{eqnarray*}
 \end{lemma}
 \begin{proof}
 We mention that $R$ is considered as the remainder term,  generated by the blowup profile $\varphi (y,s)=  \varphi_0(\frac{y}{\sqrt{s}}) + \frac{a}{s}$. Hence, it has the same structure to  the one in \cite{ZAAihn98}. For that reason,   the proof is similar to Lemma B.5 in that work,  and it stems from a Taylor expansion.  
 \end{proof}

 Next, we provide some estimates related with term  $G.$
 \begin{lemma}[Term $G$]\label{lemma-G} Let us consider $G$ defined as in \eqref{c4defini-N-term}. Then, we have the following expansion
 \begin{eqnarray}
 G(y,s) = \left(\frac{\bar{\theta}'(s)}{\bar{\theta}(s)} \right)\left(\kappa
 -\frac{\kappa b}{(p-1)^2}\frac{|y|^2}{s} \right)  + \tilde{G}, \label{estimate-G-y-s}
 \end{eqnarray}
 where
 $$ \left|\tilde{G}(y,s) \right| \le C \left| \frac{\bar{\theta}(s)}{\bar{\theta}(s)}\right| \left(\frac{1+|y|^2}{s^2} +|q(y,s)| \right), \forall |y| \le K_0 \sqrt{s},$$
 and  $a$ defined as in \eqref{defi-a}. 
 
Furthermore, if we assume that $u \in S(K_0, \epsilon_0, \alpha_0,A,\delta_0,C_0,\eta_0,t),$ with $t =-\ln(T-t)$ and  \eqref{esti-rigorous-theta-bar-}  holds, then  we have the following global bound
 \begin{eqnarray}
 \|G(.,s)\|_{L^\infty} \le \frac{C}{s}. \label{esti-G-global}
 \end{eqnarray}
 
 \end{lemma}
 \begin{proof}
 Now we consider  $|y| \le K_0 \sqrt{s}$ then via $G$'s definition we derive
 \begin{eqnarray*}
 G(y,s) =  \frac{ \bar{\theta}'(s)}{(p-1)\bar{\theta}(s)}   ( \varphi + q),
 \end{eqnarray*}
 where $ \varphi(y,s) = \varphi_0 \left( \frac{y}{\sqrt{s}} \right) + \frac{a}{s}$ and $a$ defined as in \eqref{defi-a}. Then, by a simple Taylor expansion, we immediately derive \eqref{estimate-G-y-s}. 
 
 It remains to prove \eqref{esti-G-global}. Indeed, by $G$'s definition we have
 $$ \text{Supp}(G) \subset \{K_0 s \le  |y| \le 2K_0 s  \}.$$
 We observe that, once \eqref{esti-rigorous-theta-bar-} holds, then
 $$ \left| \frac{\bar{\theta}'(s)}{\bar{\theta}(s)} (q + \varphi) \right| \le \frac{C}{s}. $$
 So, it is sufficient to prove 
 \begin{eqnarray}
 \| \tilde F(.s) \|_{L^\infty} \le \frac{C}{s}.\label{prove-tilde-F}
 \end{eqnarray}
 Let us recall   
 \begin{eqnarray*}
 \tilde{F}(y,s) = e^{-\frac{p}{p-1}s}\left( -u \theta^\frac{1}{p-1}(t) \Delta \chi_1 -2\theta^{\frac{1}{p-1}} \nabla \chi_1 \cdot \nabla u + (\theta^\frac{1}{p-1} u )^{p} (\chi_1-\chi_1^p)(x,t) \right),
 \end{eqnarray*}
 where $ s = - \ln(T-t)$ and $ y =\frac{x}{\sqrt{T-t}}$.  For all $|y| \in  [K_0s, 2K_0 s]$, 
 using Lemma  \ref{lemma-estimate-U-x-t-in-S-t} and \eqref{asymptotic-rho-rho(x)-x-t-t-ln-T-t}, we obtain the following estimates: 
 \begin{eqnarray*}
 \left|   u \theta^\frac{1}{p-1}(t) \Delta_x \chi_1 \right|  \le C \frac{(T-t)^{-\frac{p}{p-1}}}{|\ln(t-t)|},\\
 \left|    \theta^{\frac{1}{p-1}}(t) \nabla_x \chi_1 \cdot \nabla u \right|  \le C \frac{(T-t)^{-\frac{p}{p-1}}}{|\ln(t-t)|}.
 \end{eqnarray*} 
 In  addition, the definition of $\chi_1$ gives
 \begin{eqnarray*}
 \chi_1 (x,t) = \chi_0 \left(\frac{|y|}{K_0s} \right), 
 \end{eqnarray*}
 and thus for all $  \frac{|y|}{s} \in \left[ K_0, 2 K_0 \right]$, it follows  that
 $$ |\chi_1(1 -\chi_1^{p-1})| \le C, $$
 and 
 $$ | \theta^{\frac{1}{p-1}}  u | \le \frac{(T-t)^{-\frac{1}{p-1}}}{|\ln(T-t)|}. $$
 The latter implies 
 $$ \left| e^{-\frac{p}{p-1}s} (\theta^\frac{1}{p-1}(t) u )^p \right| \le \frac{C}{s}$$
 and thus \eqref{prove-tilde-F} follows.  This completes the proof of the Lemma.
 \end{proof}

\section{Parabolic estimates}\label{appendix-B}

Let us recall that  $e^{t\Delta} $ is  the semi-group generated by $\Delta$ associated with Neumann boundary conditions in problem \eqref{equa-u-non-local}.     It is proved in \cite[Lemma 3.3]{YZPA13}  that the associated  heat Kernel   $G(x,x',t)$     satisfies the estimates 
$$  \left|  \nabla_x^i  G (x,x',t) \right| \le C^{- \frac{N+1}{2}} \exp\left(  -C(\Omega) \frac{|x-x'|^2}{t} \right),$$
and
$$    e^{t\Delta} u_0 = \int_\Omega G(x,x',t) u_0(x') dx'.     $$
In particular,  we claim to the following   
\begin{lemma}[Parabolic estimates]\label{parabolic estimates} Let us consider $T,K_0, \epsilon_0, \alpha_0, A, \delta_0, C_0, \eta_0$  be positive constants such that
$$ u \in S(T,K_0, \epsilon_0, \alpha_0, A, \delta_0, C_0, \eta_0, t), \forall t \in [0,t^*).$$	 
Then,  for all  $ x \in \Omega  $, we can find $R_x >0, C =  C(A,K_0,\epsilon_0,\alpha_0,\delta_0,C_0,\eta,T)$ such that
$$ \|\partial_t u(x,t)\|_{B(x,R_x)} \le C.$$	
\end{lemma}
\begin{proof}
The proof is similar to \cite[Lemma F.1]{DZM3AS19}. We kindly refer the readers to that reference for more details.
\end{proof}

\section{Some necessary estimates  and integrals }\label{estimat-pro(x)}
In this part, we aim to give some fundamental estimates on key quantities. We also provide useful formulas of some  key integrals arise in the proofs throughout the manuscript.

We first provide some estimates  on $t(x)$:
\begin{lemma}\label{lemma-t(x)} 
Let us consider  $t(x)$, defined as in \eqref{c4defini-t(x)-} for all $|x| \leq \epsilon_0$, and $\rho(x) =  T - t(x)$. Then, we have
$$ \rho(x) = \frac{8}{K_0^2} \frac{|x|^2}{|\ln|x||} \left( 1 + \frac{|\ln |\ln|x|||}{|\ln|x||} \right),  $$
and 
$$ \ln \rho (x) \sim 2 \ln|x| \left(1 + \frac{|\ln |\ln|x|||}{|\ln|x||} \right),\quad\mbox{as}\quad x\to 0.$$
 In particular, if $|x| = K_0\sqrt{T-t}|\ln(T-t)|$, then
\begin{equation}\label{asymptotic-rho-rho(x)-x-t-t-ln-T-t}
   \rho(x) = 16 (T-t)|\ln(T-t)| \left( 1+ O\left(\frac{|\ln |\ln(T-t)||}{|\ln(t-t)|} \right)\right), \text{ as } t \to T.  
\end{equation}
\end{lemma}
\begin{proof}
The proof directly  follows  by \eqref{c4defini-theta}.
\end{proof}

\begin{lemma}[Bubble integrals]\label{bubles-integral}
Let us consider $N >0,k \in \N^*, p>1,$ and $b>0$, we now define
$$ I_{b,p,N,k} = \int_0^\infty (p-1 + b\xi^2 )^{-k-\frac{N}{2}} \xi^{N-1} d\xi.$$
Then, for all $k\ge 1$, we have
\begin{eqnarray}
I_{b,p,N,1} &=& \frac{1}{(p-1) N} b^{-\frac{N}{2}}\label{integral-I-k=1},\\
I_{b,p,N,k+1}  & = & \frac{1}{p-1} \frac{k}{k+\frac{N}{2}} I_{b,p,N, k}.\label{integral-I-k-ge-1}
\end{eqnarray}
\end{lemma}
\begin{proof}
The result follows the integration by parts. Indeed, we firstly handle \eqref{integral-I-k-ge-1}. We consider $k\ge 1$,  then,   we write
\begin{equation*}
    I_{b,p,N,k}   =  \int_{0}^\infty (p-1 +b \xi^2)^{-k -\frac{N}{2}} \xi^{N-1} d\xi .  
\end{equation*}
Using integral by parts with $u = (p-1+ b\xi^2)^{-k-\frac{N}{2}}, dv = \xi^{N-1} d\xi $, we obtain 
\begin{eqnarray*}
I_{b,p,N,k} & = & \frac{2}{N}\left( k +\frac{N}{2} \right) \int_0^\infty (p-1+b\xi^2)^{-k-1-\frac{N}{2}} b \xi^2 d\xi\\
& =  &  \frac{2}{N}\left( k +\frac{N}{2} \right) \left[  \int_0^\infty (p-1 + b \xi^2)^{ -k -\frac{N}{2}} \xi^{N-1}d\xi -(p-1)\int_0^\infty (p-1 +b\xi^2)^{-k-1-\frac{N}{2}} \xi^{N-1} d\xi     \right] \\
& = & \frac{2}{N}\left( k +\frac{N}{2} \right) ( I_{b,p,N,k} - (p-1) I_{b,p,N,k+1},
\end{eqnarray*}
which concludes 
\begin{eqnarray*}
 I_{b,p,N,k+1} = \frac{1}{p-1} \frac{k}{k+\frac{N}{2}} I_{b,p,N,k}. 
\end{eqnarray*}
Thus, \eqref{integral-I-k-ge-1} follows. 

\medskip
Next, we prove to \eqref{integral-I-k=1}. Let $K_0$ be an arbitrary positive constant and we consider the following integral
$$ \int_0^{K_0}  (p-1 + b\xi^{2})^{-\frac{N}{2}} \xi^{N-1} d\xi.    $$
Using integration by parts, we obtain 
\begin{eqnarray*}
\int_{0}^{K_0} (p-1 + b\xi^{2})^{-\frac{N}{2}} \xi^{N-1} d\xi &=& \frac{\xi^N}{N} (p-1 +b\xi^2)^{-\frac{N}{2}} \left. \right|_0^{K_0} - (p-1) \int_0^{K_0} (p-1+b\xi^2)^{-\frac{N}{2} - 1} \xi^{N-1} d\xi \\
&+&     \int_{0}^{K_0} (p-1 + b\xi^{2})^{-\frac{N}{2}} \xi^{N-1} d\xi. 
\end{eqnarray*}
Taking $K_0  \to +\infty$,  we obtain
\begin{eqnarray*}
I_{b,p,N,1} =  \frac{1}{p-1} \frac{1}{N b^\frac{N}{2}}. 
\end{eqnarray*}
Finally, we conclude  \eqref{integral-I-k=1}.
\end{proof}

\newpage

\bibliographystyle{alpha}
\bibliography{mybib}

\end{document}